\documentclass[10pt]{amsart}
\usepackage[margin=1in,letterpaper,portrait]{geometry}
\usepackage{enumitem}
\usepackage{graphicx}
\usepackage{color}
\usepackage{latexsym,amsmath,amssymb,verbatim}
\usepackage{mathtools}

\theoremstyle{plain}
    \newtheorem{theorem}[equation]{Theorem}
   \newtheorem*{theorem*}{Theorem} 
      \newtheorem{proposition}[equation]{Proposition}
   \newtheorem{prop}[equation]{Proposition}
   \newtheorem{lemma}[equation]{Lemma}
   \newtheorem{corollary}[equation]{Corollary}
   \newtheorem{conjecture}[equation]{Conjecture}
   
\theoremstyle{definition}
   \newtheorem{definition}[equation]{Definition}

   \newtheorem{remark}[equation]{Remark}

\numberwithin{equation}{section}

\newcommand{\cchi}{\fontdimen16\textfont2=3pt
\fontdimen17\textfont2=3pt \chi}

\newcommand\qbin[3]{\left[\begin{matrix} #1 \\ #2 \end{matrix} \right]_{#3}}
\newcommand{\ot}{\otimes}

\newcommand{\CC}{{\mathbb {C}}}

\newcommand{\QQ}{{\mathbb {Q}}}
\newcommand{\ZZ}{{\mathbb {Z}}}

\newcommand{\Hilb}{{\operatorname{Hilb}}}

\newcommand{\Frac}{{\operatorname{Frac}}}

\newcommand{\Cat}{{\operatorname{Cat}}}

\newcommand{\supp}{{\operatorname{supp}}}

\newcommand{\del}{\partial}

\newcommand{\xx}{\mathbf{x}}
\newcommand{\om}{\omega}

\newcommand{\SSS}{{S}}

\newcommand{\symm}{\mathfrak{S}}

\begin{document}

\title[Invariant theory for coincidental complex reflection groups]
{Invariant theory for coincidental complex reflection groups}
\date{\today}

\author{Victor\ Reiner}
\author{Anne V.\ Shepler}
\author{Eric Sommers}
\email{reiner@math.umn.edu, ashepler@unt.edu, esommers@math.umass.edu}

\thanks{First author partially supported by NSF grant DMS-1601961; second author partially supported by
Simons Foundation Grant \#429539.}



\begin{abstract}
  V.F.\ Molchanov considered the Hilbert series for the
  space of invariant skew-symmetric tensors and dual tensors with polynomial coefficients
  under the action of a real reflection group,
  and he speculated that it had a certain product formula
  involving the exponents of the group.
 We show that Molchanov's speculation is false in general
but
holds for all {\it coincidental} complex reflection groups
 when appropriately modified using exponents and co-exponents.
These are the irreducible
well-generated (i.e., duality) reflection groups with exponents forming an arithmetic progression
and include many real reflection groups and all non-real Shephard groups, e.g., 
the Shephard-Todd infinite family $G(d,1,n)$.
We highlight consequences for the $q$-Narayana and $q$-Kirkman polynomials,
giving simple product formulas for both, and
give a $q$-analogue of the identity transforming the $h$-vector to the $f$-vector for the coincidental
finite type cluster/Cambrian complexes of Fomin--Zelevinsky and Reading.
We include the determination of the Hilbert series for the non-coincidental irreducible complex reflection groups as well.
\end{abstract}

\maketitle

\vspace{-1ex}
\section{Introduction}
\vspace{-1ex}
Molchanov~\cite{Molchanov} hypothesized a formula for the dimensions 
of invariants of certain finite real reflection groups acting on skew-symmetric tensors and
dual tensors with polynomial coefficients.
His formula gives evidence that Solomon's invariant theory~\cite{Solomon}
for differential forms may have an extension
to mixed derivation differential forms.
We examine these forms not just for real reflection groups, but
for complex reflection groups in general and compute their Hilbert series.
We reformulate Molchanov's hypothesis
in terms of exponents and coexponents of the group.
Although his formula and this reformulation
do not hold for all real reflection groups,
they do hold for the important class of {\em coincidental reflection groups}.

The invariant theory
of reflection groups acting on $V=\CC^n$ displays a wondrous numerology 
controlled by two sequences of positive integers, the {\it exponents} $e_1 \leq e_2  \leq \cdots \leq e_n$ 
and {\it coexponents} $e^*_1 \leq e^*_2  \leq \cdots \leq e^*_n$.
Solomon's Theorem~\cite{Solomon} gives the dimensions of
$W$-invariant polynomial differential forms on $V$ entirely in terms of the exponents
of the reflection group $W$; 
his proof extends to describe likewise
the invariant derivation forms in terms
of the coexponents (see~\cite{OrlikSolomon}).
The {\it degrees} are the integers $d_i=e_i+1$.
Those reflection groups satisfying $e_i+e^*_{n+1-i} = d_n$ are called
{\em duality groups}. These are precisely the {\em well-generated} reflection groups, i.e., those
generated by $n=\dim V$ reflections, and include all Coxeter groups.

An irreducible duality group $W$ is {\em coincidental} if its exponents $(e_1, e_2, \ldots, e_n)$ form an arithmetic sequence $(e_1, e_1+a, e_1+2a,\ldots,e_1+(n-1)a)$ for some positive integer $a$
which we call its {\it exponent gap}.  The coincidental reflection groups are the Coxeter groups of types $A_n$, $B_n/C_n$, $I_2(m)$, $H_3$, the monomial groups $G(d,1,n)$, all irreducible
duality groups in rank $2$, and the groups 
$G_{25}$, $G_{26}$, and $G_{32}$
in the Shephard-Todd classification~\cite{ShephardTodd}.
They include all non-Coxeter {\it Shephard groups}, that is, the symmetry groups of regular complex polytopes.

We extend Solomon's description for the Hilbert series 
for invariant differential forms, $(S(V^*)\ot \wedge V^*)^W$,
to invariant {\it mixed derivation differential forms},
$(S(V^*)\ot \wedge V^*\ot \wedge V)^W$.  
We use the {\it elementary symmetric functions}
$
\sigma_r(x_1,\ldots,x_n):=\sum_{1\leq i_1 < \cdots < i_r \leq n} x_{i_1} x_{i_2} \cdots x_{i_r} 
$
with the convention that $\sigma_0(x_1,\ldots,x_n)\equiv 1$.

\begin{theorem}
\label{maintheorem}
For any coincidental complex reflection group $W$ acting on $V=\CC^n$, 
$$
\Hilb\left((S(V^*)\otimes \wedge V^* \otimes \wedge V)^W,q,t ,s \right) 
\ =\  \sum_{r=0}^n \ s^r\  \sigma_r(q^{e_1^*}, \ldots, q^{e_n^*})
  \  \frac{
  \prod_{i=1}^{r} (1+q^{-e^*_i}t)\
   \prod_{i=1}^{n-r} (1+q^{e_i}t) }
      { \prod_{i=1}^{n} (1- q^{d_i})  }
$$
where the coefficient of $q^i t^k s^r$ in the Hilbert series  
is the dimension of $(S^i(V^*) \otimes \wedge^k V^* \otimes \wedge^r V)^W$.
\end{theorem}

Here we use the standard grading
on $S(V^*)=\oplus_i\, S^i(V^*)$ by polynomial degree.
We may reformulate Theorem~\ref{maintheorem} compactly
using the {\it $q$-Pochhammer} notation defined by
$$
(z;q)_k:=(1-z)(1-zq)\cdots(1-zq^{k-1})
$$
and the {\it $q$-binomial coefficient} defined by
\begin{equation}
\label{q-binomial-definition}
\qbin{n}{r}{q}:=\frac{ (q;q)_n}{(q;q)_r \, (q;q)_{n-r}}\, .
\end{equation}
For a coincidental reflection group $W$, since the coexponents
$(e_1^*, e_2^*\ldots, e_n^*)=(1, 1+a, 1+2a, \ldots, 1+(n-1)a)$,
we can use the well-known  identity~\cite[Chap.~I \S 2, Ex.~3]{Macdonald}
\begin{equation}
\label{principal-specialization-of-elementary}
\sigma_r(1,q,q^2,\ldots,q^{n-1})=q^{\binom{r}{2}}\qbin{n}{r}{q}
\end{equation}
to rewrite the $r^{\text{th}}$ elementary symmetric function 
appearing in the theorem as
$$
\sigma_r( q^{e^*_1}, \ldots, q^{e^*_n} ) 
= \sigma_r( q^1,q^{1+a},q^{1+2a}, \ldots, q^{1+(n-1)a} )
= q^r \cdot \sigma_r(q^a,q^{2a},\ldots,q^{(n-1)a})
=q^{r+a\binom{r}{2}} \qbin{n}{r}{q^a}\, .
$$
We focus on the summand $\wedge^r V$ in $\wedge V=\oplus_{r=0}^n \wedge^r V$
and give the following equivalent version of Theorem~\ref{maintheorem}.

\vspace{2ex}

\noindent
{\bf Theorem~\ref{maintheorem}${}^\prime$}.
{\it
For a coincidental complex reflection group $W$ with smallest exponent $e_1$, exponent gap $a$,
$$
\Hilb\left((S(V^*)\otimes \wedge V^* \otimes \wedge^r V)^W,q,t \right) 
= q^{r+a\binom{r}{2}} 
   \qbin{n}{r}{q^a}  
    \frac{(-tq^{e_1};q^a)_{n-r}  (-tq^{-1};q^{-a})_{r}}
               {(q^{e_1+1};q^a)_n}\, 
           \quad\text{for } r=0, \ldots, n.
$$
}

In fact, we compile the data on the Hilbert series of $(S(V^*)\ot \wedge V^*\ot \wedge V)^W$
for {\it all} irreducible complex reflection groups $W$,
not just the coincidental groups---see Section~\ref{G(d,1,n)-section} for Shephard and Todd's infinite family $G(de,e,n)$ of monomial
groups and Section~\ref{data-section} for the exceptional groups.

\subsection*{The $q$-analogues of $f$-vectors and $h$-vectors}
Theorem~\ref{maintheorem}${}^\prime$ suggests
$q$-analogues of the {\it $f$-vector} and the {\it $h$-vector}
appearing in the algebraic combinatorics of certain simple polytopes
and simplicial spheres, as we will explain in Section~\ref{parking-section}.
These vectors record the number of faces of each dimension
and the Hilbert-Poincar\'e polynomial of the associated Stanley-Reisner ring.
For a coincidental reflection group $W$, we rename the right side of Theorem~\ref{maintheorem}${}^\prime$
as follows:
$$
f_r(W;q,t)\ :=\
q^{r+a\binom{r}{2}} 
   \qbin{n}{r}{q^a}  
    \frac{(-tq^{e_1};q^a)_{n-r}  (-tq^{-1};q^{-a})_{r}}
               {(q^{e_1+1};q^a)_n}\,.
 $$
 We wish to relate $f_r(W;q,t)$ to a second product:
 $$
 h_r(W;q,t) := (-tq^{-ar-1})^{n-r} \qbin{n}{r}{q^a}
                    \frac{ (-tq^{-1};q^{-a})_r } { (q^{e_1+1};q^a)_r }\, .
$$ 
We will see in Section~\ref{parking-section} that the specializations
$$f_r:=\left[ f_r(W;q,q^{h+1}) \right]_{q=1}
\quad\text{for Coxeter number}\quad h:=e_n+1
$$
give the number of faces of each dimension in the {\it finite type cluster
fans} of Fomin and Zelevinsky~\cite{FominZelevinsky} when $W$ is a Weyl group, 
or the {\it Cambrian fans} of Reading~\cite{Reading} when $W$ is a real reflection group. 
For simplicial fans or polytopes,
a standard re-encoding gives the {$f$-vector} entries $f_r$
in terms of the {$h$-vector} entries $h_r$:
\begin{equation}
\label{usual-h-to-f-transformation}
\sum_{r=0}^n s^r f_r  \
=\ \sum_{r=0}^n (1+s)^r \cdot h_r\, .
\end{equation}
In Section~\ref{h-to-f-section}, we use Theorem~\ref{maintheorem}${}^\prime$ 
to prove a $q$-analogue (and even a $(q,t)$-analogue) of this standard encoding:
\begin{theorem}
\label{h-to-f-theorem}
For any coincidental reflection group $W$ with exponent gap $a$, 
$$
\sum_{r=0}^n s^r f_r(W;q,t)  \
=\ \sum_{r=0}^n (-sq;q^a)_r \cdot h_r(W;q,t).
$$
\end{theorem}
In Section~\ref{parking-section}, we explain why specializing $t$ in $f_r(W;q,t)$ and $h_r(W;q,t)$ to certain powers of $q$ 
give the {\it $q$-Catalan numbers}, {\it $q$-Kirkman numbers}, and {\it $q$-Narayana numbers} arising previously in~\cite{ArmstrongReinerRhoades, ReinerSommers, Sommers} and how these specialize further to the aforementioned $f$-vector and $h$-vector entries.

\subsection*{Outline}
After recalling the numerology of reflection groups in Section~\ref{background},
we show Theorem~\ref{maintheorem} directly 
for the infinite family $G(d,1,n)$ and the Weyl groups of type $A$ in Section~\ref{G(d,1,n)-section}
using results of Kirillov and Pak~\cite{KirillovPak} and Koike \cite{Koike}.
We also give the Hilbert series explicitly for the groups $G(de, e,n)$
in Section~\ref{G(d,1,n)-section}.
We conjecture an explicit basis for $(S(V^*)\otimes \wedge V^* \otimes \wedge V)^W$
in Section~\ref{conjecture} constructed from invariant differential operators
for coincidental reflection groups; invariance 
of the alleged basis elements in Conjecture~\ref{structural-Molchanov}
is checked in Section~\ref{invariance-of-operators-section}.
In Section~\ref{Gutkin-Opdam-section}, we use the Gutkin-Opdam Lemma to predict the
sum of degrees of these alleged basis elements.
Section~\ref{mainproof} then outlines the proof  of Theorem~\ref{maintheorem}
and compares it to Molchanov's original hypothesis.
It also explains how we used {\tt Mathematica} to verify 
Conjecture~\ref{structural-Molchanov}  for the
real reflection group $H_3$ and the Shephard groups $G_{25}$, $G_{26}$, $G_{32}$.
We verify Conjecture~\ref{structural-Molchanov} 
for rank~$2$ groups in Section~\ref{two-dimensional-section}.
In Section~\ref{h-to-f-section}, we use Theorem~\ref{maintheorem}
to define the above $q$-analogues of the $f$-vector and $h$-vector, and
we prove Theorem~\ref{h-to-f-theorem} giving a $(q,t)$-analogue of
the transformation \eqref{usual-h-to-f-transformation} that converts between $f$ and $h$.
We explain how specializations of these $q$-analogues give
known product formulas for $q$-Catalan, 
$q$-Kirkman, and $q$-Narayana numbers
in Section~\ref{parking-section}
and explain connections to graded {\it parking spaces}.
Lastly, in Section~\ref{data-section}, we compile the Hilbert series of
$(S(V^*)\ot \wedge V^* \wedge V)^W$
for all of the exceptional irreducible complex reflection groups.

\section{Invariant theory of reflection groups}
\label{background}
We begin by recounting some appearances of the {\em exponents} and 
{\em coexponents} in the invariant theory of reflection groups.  
Recall that a  {\em reflection} on $V=\CC^n$ is a linear transformation
whose fixed point space is a hyperplane and a  
{\it reflection group} $W$ is a subgroup of $\text{GL}(V)$ 
generated by reflections.  We assume all reflection groups are finite.
Consequently,
we may take an inner product on $V$ with respect to which $W$ acts by isometries
and fix a basis of $V$
so that the matrices giving the action are unitary.
We write $\det=\det_V$ throughout for the determinant of elements of $W$ acting
on $V$.
A reflection group is a {\em (finite) Coxeter group} or {\em real reflection group}
if it is generated by reflections on $\mathbb{R}^n$, which then act on $\CC^n$ by extension of scalars.

A large body of literature describes the invariant theory
of reflection groups acting on $V=\CC^n$ in terms of
two sequences of positive integers, the {\it exponents} $e_i$ 
and {\it coexponents} $e^*_i$ of $W$,
\begin{equation}
\label{exponent-coexponent-indexing}
\begin{aligned}
e_1 \leq e_2  \leq \cdots \leq e_n
\quad\text{and}\quad
e^*_1 \leq e^*_2  \leq \cdots \leq e^*_n,
\end{aligned}
\end{equation}
defined as follows.  The dual action of $W$ on $V^*$ induces an action
on the {\it symmetric algebra}
$$
S(V^*) \cong \CC[x_1,\ldots,x_n]
$$
where $x_1,\ldots,x_n$ is the $\CC$-basis for $V^*$
dual to a $\CC$-basis $y_1,\ldots,y_n$ of $V$; the group $W$ acts on $S(V^*)$
via invertible linear substitutions of the variables $x_1,\ldots,x_n$.
A theorem of Shephard and Todd~\cite{ShephardTodd} and of Chevalley~\cite{Chevalley}
asserts that the $W$-invariant polynomials form a polynomial subalgebra:
$$
S(V^*)^W=\CC[f_1,\ldots,f_n]
$$ 
for some homogeneous $f_i$ in $S(V^*)$ called {\em basic invariants}.  Their
polynomial degrees $d_1 \leq \cdots \leq d_n$  are independent of the choice
of the $f_i$.
The exponents of $W$ are then just the integers $e_i:=d_i-1$.  

More generally, we may define $U$-exponents for any $W$-representation $U$
by regarding the $W$-fixed space $\left(S(V^*) \otimes U\right)^W$ as a
module over $S(V^*)^W$
via multiplication into the left tensor factor.
This module
is free of rank $\dim U$
by Chevalley's Theorem~\cite{Chevalley} (see ~\cite[Prop.~4.3.3, eqn.~(4.6)]{Broue})
or by a result of Hochster and Eagon~\cite{HochsterEagon}, and
the {\it $U$-exponents}
$
e_1(U) \leq \cdots \leq e_{\dim U}(U)
$
are the degrees of a homogeneous basis.
Here, $S(V^*) \otimes U$ inherits the grading on $S(V^*)$ by polynomial degree.
Note that these $U$-exponents are the degrees in which the representation $U^*$
appears in the coinvariant algebra $S(V^*)/S(V^*)_+^W$.

As a special case, the $V$-exponents are the coexponents $e_i^*=e_i(V)$.
In other words, $(S(V^*) \otimes V)^W$ is a free module over $S(V^*)^W$ and one may
choose a basis $\{\theta_1,\ldots,\theta_n\}$, called a set of {\it basic derivations}, 
with
\begin{equation}
\label{basic-derivations}
\theta_i =\sum_{j=1}^n \theta_i^{\, j} \otimes y_j
\quad\text{for homogeneous } \theta_i^{\, j} 
\text{ in } S(V^*) \text{ of degree } e^*_i
\, .
\end{equation}
When $W$ is irreducible, there is a unique smallest coexponent $e_1^*=1$ corresponding to the {\it Euler derivation},
$
\theta_1=\theta_E:=\sum_{i=1}^n x_i \otimes y_i,
$
which is always $W$-invariant
(see~\cite{ReinerShepler}, for example).

Solomon~\cite{Solomon} considered the space of differential forms 
and showed that the {\it exterior algebra} $\wedge V^*$ tensored
with $S(V^*)$ has $W$-fixed space
which is an exterior algebra over $S(V^*)^W$ on exterior generators
$\{df_1,\ldots,df_n\}$:
$$
(S(V^*)\ot\wedge V^*)^W = \bigwedge_{S(V^*)^W}\{df_1, \ldots, df_n\}
\quad\text{ where }\quad
df:=\sum_{i=1}^n \tfrac{\partial f}{\partial x_i} \otimes x_i.
$$
From this one can deduce that the exponents are alternatively defined as the $V^*$-exponents via $e_i:=e_i(V^*)$. Orlik and Solomon~\cite[Thm.~3.1]{OrlikSolomon}
generalized Solomon's Theorem,
implying as a special case that
the exterior algebra $\wedge V$ tensored
with $S(V^*)$ has $W$-fixed space 
which is also an exterior algebra over base ring $S(V^*)^W$, this time
with exterior generators given by the basic derivations
$\{\theta_1,\ldots,\theta_n\}$ in $(S(V^*) \otimes V)^W$:
\begin{equation}
\label{OrlikSolomon}
(S(V^*)\ot\wedge V)^W = \bigwedge_{S(V^*)^W}\{\theta_1, \ldots, \theta_n\}\, .
\end{equation}

\subsection*{Hilbert Series}
The above structural results imply combinatorial descriptions for various {\it Hilbert series}
$\Hilb(M,q):=\sum_{d \geq 0} \dim M_d \cdot q^d$ for graded (or doubly or triply graded)
vector spaces $M=\bigoplus_{d \geq 0} M_d$.  In each case, one compares 
the Hilbert series implied by variants on
Molien's Theorem (see~\cite{Stanley1979}
or~\cite[Lemma~3.2.8]{Broue})
to the expression implied by the above results on
the structure of the various rings and modules.
Appreviating $S=S(V^*)$, we observe the following.  
\begin{itemize}
\item The Shephard-Todd-Chevalley Theorem on $S^W$ implies that
\begin{equation}
\label{Shephard-Todd-Chevalley-consequence}
\tfrac{1}{|W|} 
\sum_{w \in W} \frac{1}{\det(1-qw)}
=\Hilb(S^W,q)
=   \frac{ 1 }{\prod_{i=1}^{n} (1- q^{d_i})  }\, .
\end{equation}
\item
The definition of coexponents in terms of $(S \otimes V)^W$ implies that
\begin{equation}
\label{coexponent-definition-consequence}
\tfrac{1}{|W|} 
\sum_{w \in W} \frac{\cchi_V(w^{-1}) }{\det(1-qw)}
=\Hilb((S\otimes V)^W,q)
=\Big( \sum_{i=1}^n q^{e^*_i} \Big) \frac{1}{ \prod_{i=1}^{n} (1- q^{d_i})  }\, .
\end{equation}
\item
  Solomon's Theorem describing $(S\otimes \wedge V^*)^W$
  implies this generalization of \eqref{Shephard-Todd-Chevalley-consequence}:
\begin{equation}
\label{Solomon-consequence}
\tfrac{1}{|W|} 
\sum_{w \in W} \frac{\det(1+tw)}{\det(1-qw)}
=\Hilb((S\otimes \wedge V^*)^W,q,t)
=\prod_{i=1}^{n}\frac{1+q^{e_i}t}{1- q^{d_i}}\, ,
\end{equation}
where the coefficient of $q^i t^k$ in the
Hilbert series is the dimension of $(S^i \otimes \wedge^k V^*)^W$.
\item
  The Orlik-Solomon Theorem describing $(S\otimes \wedge V)^W$
  implies this different generalization of \eqref{Shephard-Todd-Chevalley-consequence}:
\begin{equation}
\label{Orlik-Solomon-consequence}
\tfrac{1}{|W|} 
\sum_{w \in W} \frac{\det(1+sw^{-1})}{\det(1-qw)}
=\Hilb((S\otimes \wedge V)^W,q,t)
=\prod_{i=1}^{n} \frac{1+q^{e_i^*}s}{1- q^{d_i}}\, ,
\end{equation}
where the coefficient of $q^i s^r$ in the
Hilbert series is the dimension of $(S^i \otimes \wedge^r V^*)^W$.
\end{itemize}

\subsection*{Duality groups}
More recently, the first two authors~\cite{ReinerShepler} proved
a structural statement in the invariant theory
of  {\it duality groups}, that is,
reflection groups satisfying
the {\it exponent-coexponent duality}
$$
e_i + e^*_{n+1-i}=h:=\max\{d_i\}.
$$
For any reflection group $W$, the space $(S(V^*) \otimes \wedge V^* \otimes V)^W$ is
a module over the $S(V^*)^W$-exterior algebra 
$$
(S(V^*) \otimes \wedge V^*)^W=\bigwedge_{S(V^*)^W}\{df_1,\ldots,df_n\}
$$
via multiplication in the first two tensor factors.  In general, it is {\it not free} as a  module 
over this exterior algebra.  But when $W$ is a duality group, 
$(S(V^*) \otimes \wedge V^* \otimes V)^W$ is free as a module over
the {\em subalgebra}
$$
\bigwedge_{S(V^*)^W}\{df_1,\ldots,df_{n-1}\}
$$
which omits the last exterior generator $df_n$.  This similarly implies a combinatorial
Hilbert series:
\begin{theorem}[\cite{ReinerShepler}]
For $W$ a duality group,
\begin{equation}
\label{Reiner-Shepler}
\begin{aligned}
\tfrac{1}{|W|} 
\sum_{w \in W} \cchi_V(w^{-1}) \frac{\det(1+tw)}{\det(1-qw)}
&\ =\ \Hilb\big((S(V^*)\otimes \wedge V^* \otimes V)^W,q,t\big)
\\
&\ =\    \Big( \sum_{i=1}^n q^{e^*_i}\Big) \
\frac{ \rule[-1ex]{0ex}{2ex}
(1+q^{-1}t)\ \prod_{i=1}^{n-1} (1+q^{e_i}t) }
{\rule{0ex}{2ex}
 \prod_{i=1}^{n} (1- q^{d_i})  }\, .
\end{aligned}
\end{equation}
\end{theorem}

Note that setting $t=0$ in~\eqref{Reiner-Shepler} gives~\eqref{coexponent-definition-consequence}
for duality groups $W$.

\subsection*{The $W$-invariants of
  $S(V^*) \otimes \wedge V^* \otimes \wedge^n V$}
The following theorem holds for all complex reflection groups
and agrees with extraction of the coefficient of $s^n$ in Theorem~\ref{maintheorem}.
It follows from results of the second author~\cite{Shepler99, Shepler05}, but we include a proof
since we have not seen it stated explicitly in the literature.

\begin{theorem}
\label{Shepler-product}
For any complex reflection group $W$ acting on $V$,
$$
\tfrac{1}{|W|} 
\sum_{w \in W} \det(w^{-1}) \frac{\det(1+tw)}{\det(1-qw)} 
=  \Hilb\big((S(V^*)\otimes \wedge V^* \otimes \wedge^n V)^W,q,t \big)
\, 
=\prod_{i=1}^n \frac{q^{e_i^*}+t}{1-q^{d_i}}
  \, .
  $$
\end{theorem}
 
\begin{proof}
    Since $\wedge^n V$ carries the determinant character $\det$ of $W$,
    the $S^W$-module $(S \ot \wedge V^* \ot \wedge^n V)^W$
    for $S=S(V^*)$
   is
  the space of {\it $\det^{-1}$-relative invariants} for $W$ acting
  on $S \ot \wedge V^*$.  It was shown in~\cite{Shepler99} (see also~\cite{SheplerTerao})
  that the
  $\det^{-1}$-relative invariant $1$-forms $\left(S\ot V^*\right)^{\det^{-1}}$
  form a free $S^W$-module and any basis $\{ \omega_1,\ldots,\omega_n\}$
  generates the $\det^{-1}$-relative invariant $p$-forms
  $\left( S \ot \wedge^p V^*\right)^{\det^{-1}}$ 
  freely over $S^W$ with basis
  $$
  \left\{ \tfrac{1}{Q^{p-1}\rule{0ex}{1.5ex}} \cdot \omega_{i_1} \wedge \cdots \wedge \omega_{i_p}
       : 1 \leq i_1 < \cdots < i_p \leq n \right\}.
     $$
     Here $Q:=\prod_H \ell_H$ is the $\det^{-1}$-relative invariant
     in $S$ of lowest degree,
     namely, the product of the linear forms $\ell_H$ whose vanishings
     define the reflecting hyperplanes $H$ of $W$.  Thus $Q$ has degree
     equal to the number $N^*$ of
     reflecting hyperplanes, with known formula (e.g.,~\cite[\S4.5.5, Remark~4.48]{Broue})
     $$
     \deg(Q)=:N^*=e_1^*+\cdots+e_n^*.
     $$
  If $\omega_1,\ldots,\omega_n$ are homogeneous with polynomial degrees $m_1,\ldots,m_n$,
  (so each $\omega_i$ lies in $S^{m_i} \otimes V^*$), then
  $$
  \begin{aligned}
  \frac{\Hilb\left((S\otimes \wedge V^* \otimes \wedge^n V)^W,q,t \right)}
  {\Hilb(S^W,q)}
  &= q^{\deg(Q)} \sum_{p=0}^n t^p \sum_{1 \leq i_1 < \cdots < i_p \leq n}
      q^{m_{i_1}+\cdots+m_{i_p}-p\deg(Q)}\\
      &=q^{\deg(Q)} \prod_{i=1}^n (1+tq^{m_i-\deg(Q)})
      =\prod_{i=1}^n (q^{e_i^*}+tq^{e_i^*+m_i-\deg(Q)}).      
  \end{aligned}
  $$
  Since $\Hilb(S^W,q)=\prod_{i=1}^n \tfrac{1}{1-q^{d_i}}$, the theorem will follow
  if we show that one can index so that
  $$
  m_i=\deg(Q)-e_i^*=(e_1^*+\cdots+e_n^*)-e_i^*
  $$
  for $i=1,2,\ldots,n$.
  To see this, we proceed as in the proof of~\cite[Cor.~4]{Shepler05}.  First note that the nondegenerate pairing
  $
  V \otimes \wedge^{n-1}V \rightarrow \wedge^n V
  $
  is $W$-equivariant, where $\wedge^n V$ carries the character $\det$.
  This implies $\wedge^{n-1} V \cong V^* \otimes \det$ as $W$-representations.
  Therefore the
  $S^W$-module of $\det^{-1}$-relative invariants in $S \otimes V^*$,
  which is isomorphic to the $W$-invariants
  $
  \left( S \otimes V^* \otimes \det\right)^W,
  $
  is also isomorphic to the $W$-invariants
  $\left( S \otimes\wedge^{n-1} V \right)^W$.
  However, by~(\ref{OrlikSolomon}),
  the latter has $S^W$-basis
      $\{ \theta_{1} \wedge \cdots \wedge \hat{\theta}_{i} \wedge
      \cdots \wedge \theta_n: i=1,2,\ldots,n \}$,
      whose elements indeed have degrees $(e_1^*+\cdots+e_n^*)-e_i^*$.
\end{proof}

\subsection*{Invariant derivation differential forms}
Theorem~\ref{maintheorem} describes the (triply-graded) Hilbert series
\begin{equation}
\label{molien3-var}
\Hilb\left((S(V^*) \otimes \wedge V^* \otimes \wedge V)^W,q,t,s \right)
=\tfrac{1}{|W|} \sum_{r=0}^{n} \sum_{w \in W} \frac{\det(1+tw)\det(1+sw^{-1})}{\det(1-qw)},\\
\end{equation}
where the coefficient of $q^i t^k s^r$ is the dimension of $(S^i(V^*) \otimes \wedge^k V^* \otimes \wedge^r V)^W$, in terms of
exponents and coexponents and specializes to all of
\eqref{Shephard-Todd-Chevalley-consequence},
\eqref{coexponent-definition-consequence}, \eqref{Solomon-consequence},
\eqref{Orlik-Solomon-consequence}, \eqref{Reiner-Shepler} and
Theorem~\ref{Shepler-product}.  However, it applies only
to the subfamily of {\it coincidental} reflection groups, described further here.

\vskip.1in
\noindent
{\bf Coincidental groups.}
A reflection group $W$ is {\it coincidental} if it is an irreducible duality group whose
exponents (or equivalently, its degrees, or coexponents, or codegrees) 
form an arithmetic sequence.  
The coincidental reflection groups comprise
(using notation from the classification of Shephard and Todd~\cite{ShephardTodd})
\begin{itemize}
\item the real reflection groups $A_n$, $B_n/C_n$, $I_2(m)$, $H_3$,
\item the monomial groups $G(d,1,n)$,
\item all rank $2$ duality groups, namely,
$G_4$, $G_5$, $G_6$, $G_8$, $G_9$, $G_{10}$, $G_{14}$, $G_{16}$,
$G_{17}$, $G_{18}$, $G_{20}$, and $G_{21}$, and
\item the groups $G_{25}$, $G_{26}$, and $G_{32}$.
\end{itemize}
Note that coincidental groups include all Shephard groups  that are not Coxeter groups.
Among the Coxeter-Shephard groups, coincidental groups exclude type $D_n$ for $n \geq 4$ and the real exceptional groups, $E_6$, $E_7$, $E_8$, $F_4$, $H_4$, i.e., those groups whose Coxeter diagram contains one of $D_4$, $F_4$, or $H_4$ as a subdiagram.
The coincidental groups have made multiple appearances recently, for example,
in the work of Miller~\cite[Thm.~14]{MillerFoulkes},~\cite[Thm.~2]{MillerWalls}.  
See~\cite[\S5]{Doppelgangers} for examples of real coincidental types in the literature.

Numerology of coincidental groups is governed by 
two parameters, the {\it smallest exponent} $e_1$ and the {\it gap} 
$$
a:=d_i-d_{i-1}=e_i-e_{i-1}=e^*_i-e^*_{i-1}
$$ 
between any two successive exponents, or fundamental degrees, or coexponents:
$$
\begin{array}{rcccccccl}
(d_1,d_2,\ldots,d_n)&=& ( &e_1+1,&e_1+1+a,&e_1+1+2a&,\ldots,&e_1+1+a(n-1)&)\\
(e_1,e_2,\ldots,e_n)&=& ( &e_1,&e_1+a,&e_1+2a&,\ldots,&e_1+a(n-1)&)\\
(e^*_1,e^*_2,\ldots,e^*_n)&=& (&1,&1+a,&1+2a&,\ldots,&1+a(n-1)&) \, .
\end{array}
$$
\vskip.1in

\section{Type $A$ and the monomial groups}
\label{G(d,1,n)-section}

We begin compiling our verification of Theorem~\ref{maintheorem} with the Weyl groups of type $A$ and the infinite
family of Shephard-Todd groups $G(d,1,n)$.

Recall that for positive integers $d,n$, the complex reflection 
group $G(d,1,n)$ is the
set of all $n \times n$ matrices in $\text{GL}(V)=\text{GL}_n(\CC)$
that are {\it monomial} (exactly one nonzero entry in each row and column) 
with nonzero entries
all $d^{\,\text{th}}$ roots-of-unity in $\CC$.  
Any element $w$ in $G(d,1,n)$ maps the basis vectors
$y_1,\ldots,y_n$ of $V=\CC^n$ to
$\zeta^{m_1} y_{\pi(1)},\ldots, \zeta^{m_n} y_{\pi(n)}$
for some  permutation $\pi=\pi(w)$ in the symmetric group 
$\symm_n=G(1,1,n)$, where $\zeta$ is the complex root-of-unity $e^{\frac{2\pi i}{d}}$.
In fact,  $G(d,1,n) = \symm_n \ltimes (\ZZ/d\ZZ)^n$ since
the map $w \mapsto \pi(w)$ is a surjective
group homomorphism, $G(d,1,n) {\longrightarrow}\ \symm_n$, whose
kernel is the subgroup $(\ZZ/d\ZZ)^n$ of diagonal matrices within $G(d,1,n)$.
We need to draw a distinction between the symmetric group $\symm_n$ acting as 
the permutation matrices $G(1,1,n)$ on one hand
and acting as the Weyl group $A_{n-1}$ on the other hand:
\begin{itemize} 
\item
The group $G(1,1,n)$
acts on $V=\CC^n$ {\it reducibly}, with fixed space 
$V^{\symm_n}=\CC(y_1+\cdots+y_n)$.
\item The group $W(A_{n-1})$ acts {\em irreducibly} on the quotient space
$V/\CC(y_1+\cdots+y_n)\cong \CC^{n-1}$.
\end{itemize}
 
For any finite group $W$, we introduce a shorthand notation for the Hilbert series describing the 
isotypic component in $S(V^*)\ot \wedge V^*$ corresponding to a $W$-representation 
$M$ with character $\chi$.
Again, we use a Molien type theorem to write this Hilbert series as a sum over group elements
and abbreviate $S=S(V^*)$.
We are interested in the special case when $\chi$
is the character of the $W$-representation $\wedge^r V$ giving~(\ref{molien3-var}).
\begin{definition}
For any finite subgroup $W$ of $\text{GL}(V)$ and
any character $\chi$ of a $W$-module $M$, define
$$
\begin{aligned}
P_W(\chi;q,t) &:=
\Hilb((S\ot \wedge V^*\ot M^*)^W, q,t)
=\tfrac{1}{|W|} 
\sum_{w \in W} \cchi(w)\,  \frac{\det(1+tw)}{\det(1-qw)}\, 
\quad\text{ and}\quad
\\
\cchi_{S\otimes \wedge V^*}(q,t)(w) &:=
\sum_{j=0}^\infty \sum_{k=0}^n \ q^j\, t^k \ \cchi_{S^j \otimes \wedge^k V^*}(w)
\quad\text{ for $w$ in $W$}.
\end{aligned}
$$
\end{definition}
The second expression, $\cchi_{S\otimes \wedge V^*}(q,t)(w)$, is a
class function on $W$ with values in the ring $\ZZ[t][[q]]$.
Notice that in terms of the usual inner product $\langle \, \cdot\,,\, \cdot \,\rangle_W$
on $W$-class functions, 
$$
P_W(\chi;q,t)=\langle \, \chi \, , \, \cchi_{S\otimes \wedge V^*}(q,t) \, \rangle_W\, .
$$
\noindent
  The next two subsections review 
formulas for $P_W(\chi;q,t)$ for $\chi$ a $W$-irreducible character 
for $W=W(A_{n-1})$ and $G(d,1,n)$.  Since 
$\wedge^r V$ and $\left( \wedge^r V \right)^*$ are
irreducible $W$-representations (see~\cite[\S 24-3]{Kane}), 
we examine
\begin{equation}
\label{Hilb-and-P-relation}
P_W(\left(\cchi_{\wedge^r V} \right)^* ;q,t)=
\Hilb((S\ot \wedge V^*\ot \wedge^r V)^W, q,t)
\end{equation}
to verify Theorem~\ref{maintheorem} for the
coincidental types $W(A_{n-1})$ and $G(d,1,n)$ with $d \geq 2$.

\subsection*{The Type A
formula of Kirillov-Pak, Molchanov, Thibon, Gyoja-Nishiyama-Shimura}

The following ``hook-content'' formula for $P_W(\chi;q,t)$
when $W=\symm_n=G(1,1,n)$ is the building block for the rest.  This formula was proven first by
Kirillov and Pak~\cite[eqn.~(4)]{KirillovPak} bijectively and then
more algebraically by Molchanov~\cite[eqn.~(2)]{Molchanov}.
It was also
deduced from symmetric function identities 
by Thibon~\cite[Thm.~4.3]{Thibon} and by Gyoja, Nishiyama, and Shimura~\cite[eqn.~(3.2)]{GyojaNishiyamaShimura}.
To state the formula, recall that the
irreducible characters of the symmetric group $\symm_n$ can
be indexed as $\{ \chi^\lambda \}$ where $\lambda=(\lambda_1 \geq \lambda_2 \geq \cdots \geq 0)$
runs through the partitions of $n$, that is, $|\lambda|:=\sum_i \lambda_i=n$.
Define $n(\lambda):=\sum_{i \geq 1} (i-1) \lambda_i$. 
Recall also the notion of the {\it Ferrers diagram} for $\lambda$, containing the
{\it cells} $x=(i,j)$ in row $i$ and column $j$ for $1 \leq j \leq \lambda_i$.
The cell $x=(i,j)$ is said to have {\it content} $c(x):=j-i$ and
{\it hooklength} $h(x):= \lambda_i+\#\{k: \lambda_k \geq i\}-j$.

\begin{theorem}
\cite{GyojaNishiyamaShimura, KirillovPak, Molchanov}
\label{GNS-KP-Molchanov-formula}
For $W=\symm_n=G(1,1,n)$ acting on $V=\CC^n$ and any character $\chi^\lambda$ of $W$,
$$
P_{\symm_n}(\chi^\lambda;q,t)
 = q^{n(\lambda)} \prod_{ x \in \lambda } \frac{1+t\, q^{c(x)}}{1-q^{h(x)}}\, .
$$
\end{theorem}

The theorem gives the following corollary for the irreducible action $A_{n-1}$ of the symmetric group
$\mathfrak{S}_n$.

\begin{corollary}
\label{type-A-hook-content-formula} 
For the Weyl group $W(A_{n-1})$ acting on $V=\CC^{n-1}$ and any  irreducible character $\chi^\lambda$,
\begin{align}
P_{W(A_{n-1})}(\chi^\lambda;q,t)
\label{first-type-A-formula}
 &= \frac{1-q}{1+t} \ \ P_{\symm_n}(\chi^\lambda;q,t)\\
\label{second-type-A-formula} 
 &= \frac{1-q}{1+t}  \ q^{n(\lambda)} \prod_{ x \in \lambda } \frac{1+t\, q^{c(x)}}{1-q^{h(x)}}\, .
 \end{align}
In particular,
for $\chi^{\lambda}$ the character $\cchi_{\wedge^r V}$ of the
$W$-representation $\wedge^r V$ for  some fixed $r=0,\ldots, n$, 
 \begin{equation}
\label{third-type-A-formula}
P_{W(A_{n-1})}(\cchi_{\wedge^r V};q,t)
 = q^{r+{\binom{r+1}{2}}} 
     \qbin{n-1}{r}{q} \frac{(-t\, q;q)_{n-r-1} (-t\, q^{-1};q^{-1})_r }{(q;q)_{n-1}}\, ,
\end{equation}
and the assertion of
Theorem~\ref{maintheorem} holds for type $A_{n-1}$.
\end{corollary}
\begin{proof}

  Equation~\eqref{first-type-A-formula} follows
  from the fact that
  a permutation matrix $\pi$ in $G(1,1,n)$
acts on $\CC^n$ with one extra eigenvalue $+1$ compared to its action $w$ on $\CC^{n-1}$
as an element in $W(A_{n-1})$, and hence
$$
\frac{\det(1+t\pi)}{\det(1-q\pi)}=
\frac{1+t}{1-q}\cdot  \frac{\det(1+tw)}{\det(1-qw)}
\, .
$$
Theorem~\ref{GNS-KP-Molchanov-formula} then gives~\eqref{second-type-A-formula}. 
For~\eqref{third-type-A-formula}, note that 
$\left( \cchi_{\wedge^r V}\right)^* = \cchi_{\wedge^r V} =\chi^\lambda$
for
$\lambda=(n-r,1^r)$,
which has
$$
\begin{aligned}
n(\lambda)&=\binom{r+1}{2},\\
\text{ cell contents }c(x)&=(0,1,2,\ldots,n-1-r,-1,-2,\ldots,-r),\\
\text{ hooklengths }h(x)&=(1,2,\ldots,n,1,2,\ldots,r).
\end{aligned}
$$
We apply~\eqref{second-type-A-formula} in this special case and obtain
\begin{equation}
  \label{type-A-wedge-Hilb-calculation}
\begin{aligned}
P_{W(A_{n-1})}(\left(\cchi_{\wedge^r V}\right)^*; q,t)
=P_{W(A_{n-1})}(\cchi_{\wedge^r V}; q,t) 
&=\frac{1-q}{1+t}\  \ q^{n(\lambda)} \ \prod_{ x \in \lambda } \frac{1+tq^{c(x)}}{1-q^{h(x)}} \\
&= \frac{1-q}{1+t} \ q^{\binom{r+1}{2}} \frac{1+t}{1-q^n}  \
  \prod_{i=1}^{n-r-1} \frac{1+tq^i}{1-q^i} \
  \prod_{i=1}^{r}  \frac{1+tq^{-i}}{1-q^i}\\
&=  q^{\binom{r+1}{2}} \ \frac{1-q}{1-q^n} \
  \frac{(-tq;q)_{n-r-1}}{(q;q)_{n-r-1}}
  \ \frac{(-tq^{-1};q^{-1})_r}{(q;q)_r}\\
&=  q^{\binom{r+1}{2}}\ \frac{1-q}{1-q^n} \
     \qbin{n-1}{r}{q}\ \frac{(-tq;q)_{n-r-1} \ (-tq^{-1};q^{-1})_r }{(q;q)_{n-1}} \\ 
&=  q^{r+\binom{r}{2}} \
     \qbin{n-1}{r}{q} \
    \frac{(-tq;q)_{n-r-1} \ (-tq^{-1};q^{-1})_r }{(q^2;q)_{n-1}}  
\end{aligned}
\end{equation}
using \eqref{q-binomial-definition} in the penultimate equality.

Finally, to verify 
Theorem~\ref{maintheorem} for $W(A_{n-1})$,
we apply \eqref{Hilb-and-P-relation}
and check that the last expression in \eqref{type-A-wedge-Hilb-calculation}
agrees with the expression in Theorem~\ref{maintheorem}${}^\prime$  
(see~(\ref{molien3-var})).  This holds because $W(A_{n-1})$ has rank $n-1$ and exponents 
$$(e_1,\ldots,e_{n-1})=(1,2,\ldots,n-1),$$
so that its smallest exponent is $e_1=1$ and the gap between exponents is $a=1$. 
We also use the fact that $W(A_{n-1})$ is a real reflection group, so that
$V\cong V^*$ and $e_i^*=e_i$.
\end{proof}

\subsection*{Koike's formula for $G(d,1,n)$}

We review here a formula of Koike \cite{Koike} generalizing
to the monomial groups $W=G(d,1,n)$ for $d \geq 2$
the calculation of $P_W(\chi,t)$ completed for $G(2,1,n)=W(B_n)$ 
by Kirillov and Pak~\cite[eqn.~(6)]{KirillovPak};
see also Gyoja, Nishiyama, and Shimura~\cite[eqn.~(3.9)]{GyojaNishiyamaShimura}
for the case of $W(B_n)$.

Fix $d \geq 2$, and let us abbreviate $W_n:=G(d,1,n)$ acting on $V=\CC^n$.
The irreducible characters of $W_n$ can be indexed by {\it $d$-multipartitions of $n$}
\begin{equation}
\label{multipartition}
\underline{\lambda}=(\lambda^{(0)},\lambda^{(1)},\ldots,\lambda^{(d-1)})
\end{equation}
in which each $\lambda^{(i)}$ is a partition of $n_i$ and $\sum_{i=0}^{d-1} n_i = n$.
Denoting $\underline{n}:=(n_0,n_1,\ldots,n_{d_1})$, let $W_{\underline{n}}$ be
the subgroup of $W_n$ isomorphic to $W_{n_0} \times W_{n_1} \times \cdots \times W_{n_{d-1}}$
consisting of the block diagonal matrices in $W_n$ with diagonal block sizes
specified by $\underline{n}$.  

Recall that any element $w$ in $G(d,1,n)$ can be written uniquely as  $w=\text{diag}(w) \cdot \pi(w)$, the product of a diagonal matrix $\text{diag}(w)$
and a permutation matrix $\pi(w)$ in $G(1,1,n)$.
This gives rise to a $1$-dimensional character recording the determinant of the diagonal part of $w$
(i.e., the product of the nonzero entries in $w$):
\begin{equation}
  \label{definition-of-epsilon-character}
\epsilon: W_n  \longrightarrow \CC^\times,
\quad\quad
w \mapsto  \det(\text{diag}(w))\, .
\end{equation}
Given any $\symm_n$-character $\chi$,
one can {\it inflate} it along $\pi$ to a $W_n$-character that we will
denote $\chi \Uparrow_{\symm_n}^{W_n}$.  Given any $W_{\underline{n}}$-character $\chi$,
one can {\it induce} it up to $W_n$, giving a character that we will denote
 $\chi \uparrow_{W_{\underline{n}}}^{W_n}$.  One then has the following description for the
irreducible $W_n$-character indexed by $\underline{\lambda}$ as in \eqref{multipartition}:
$$
\chi^{\underline{\lambda}} 
  = \left( 
          \bigotimes_{i=0}^{d-1} 
                    \epsilon^i \otimes 
                          \left( \chi^{\lambda_i}  \Uparrow_{\symm_{n_i}}^{W_{n_i}}\right)
    \right) \uparrow_{W_{\underline{n}}}^{W_n}.
$$
Here $\epsilon$ is the same degree $1$ character of $W_n$ restricted to each $W_{n_i}$, and
$\epsilon^i$ is its $i^{\,\text{th}}$ tensor power.

We can now state Koike's result.  We include a proof which is
shorter and less computational than the one in \cite[pp. 545-548]{Koike},
following the methodology of Kirillov and Pak.

\begin{theorem} \cite[Theorem 1]{Koike}
\label{Kirillov-Pak-generalized}
For any $d$-multipartition $\underline{\lambda}$ of $n$ as above,
$$
P_{W_n}(\chi^{\underline{\lambda}};q,t)
= P_{\symm_{n_0}}(\chi^{\lambda^{(0)}};q^d,tq^{d-1})
   \cdot \prod_{i=1}^{d-1} q^{n_i(d-i)}\cdot P_{\symm_{n_i}}(\chi^{\lambda^{(i)}};q^d,tq^{-1}).
$$
\end{theorem}
\begin{proof}
(cf.~\cite[Proof of Lemma~1]{KirillovPak})
By Frobenius reciprocity,
\begin{equation}
\label{Frobenius-reciprocity-simplification}
P_{W_n}(\cchi^{\underline{\lambda}};q,t)
= \big\langle \,\, \cchi^{\underline{\lambda}} \, , \, \cchi_{S\otimes \wedge V^*}(q,t) \,\, \big\rangle_{W_n}
= \prod_{i=0}^{d-1}
     \left\langle \,\, \epsilon^i \otimes 
                          \left( \cchi^{\lambda^{(i)}}  \Uparrow_{\symm_{n_i}}^{W_{n_i}} \right)
                 , \, \cchi_{S\otimes \wedge V^*}(q,t) \,\, \right\rangle_{W_{n_i}}.
\end{equation}
It remains to compute the $i^{\,\text{th}}$ factor in
the product.  For ease of notation, replace $n_i$ by $n$ and $\lambda^{(i)}$ by $\lambda$,
so that we may rewrite the $i^{\,\text{th}}$ factor simply as
\begin{equation}
\label{factor-rewriting}
\begin{aligned}
\left\langle \,\, \epsilon^i \otimes 
                          \left( \chi^{\lambda}  \Uparrow_{\symm_n}^{W_n} \right)
                 , \, \cchi_{S\otimes \wedge V^*}(q,t) \,\, \right\rangle_{W_n}
&=
\left\langle \,\, \chi^{\lambda}  \Uparrow_{\symm_n}^{W_n} 
                 , \, \epsilon^{-i} \otimes \cchi_{S(V^*)\otimes \wedge V^*}(q,t) \,\, 
\right\rangle_{W_n} \\
&=
\left\langle \,\, \chi^{\lambda} , \,
                 \cchi_{\left( \epsilon^{-i} \otimes 
                          S(V^*)\otimes \wedge V^* \right)^{(\ZZ/d\ZZ)^n}}(q,t) \,\, 
\right\rangle_{\symm_n}\, .
\end{aligned}
\end{equation}
Here the last equality uses a general adjointness statement 
(see, e.g.,~\cite[\S 4.1.6]{GrinbergReiner}) for quotient groups $G \rightarrow G/N$, 
taking $(G,N,G/N)$ to be $(W_n,(\ZZ/d\ZZ)^n,\symm_n)$: 
For any $G$-representation $A$, with residual $G/N$-representation on its $N$-fixed space $A^N$, and
any $G/N$-representation $B$,
$$
\langle  \,\, \cchi_B \Uparrow_{G/N}^G \, , \, \cchi_A \,\, \rangle_G 
= \langle \,\, \cchi_B \, , \, \cchi_{A^N} \,\, \rangle_{G/N} \, .
$$

We next examine, for $0 \leq i \leq d-1$, the $(\ZZ/d\ZZ)^n$-fixed space,
\begin{equation}
\label{relative-invariants-of-epsilon-power}
\left( \epsilon^{-i} \otimes S(V^*)\otimes \wedge V^* \right)^{(\ZZ/d\ZZ)^n}\, .
\end{equation}
Note the following tensor product decomposition, compatible with the $(\ZZ/d\ZZ)^n$-action:
$$
S(V^*) \otimes  \wedge V^* = \bigotimes_{j=1}^n \CC[x_j] \otimes \wedge \CC{x_j}\, .
$$
Since the character $\epsilon$ restricted to $(\ZZ/d\ZZ)^n$ is just
the tensor product of $\epsilon$ on each of the factors $\ZZ/d\ZZ$,
the $(\ZZ/d\ZZ)^n$-fixed space \eqref{relative-invariants-of-epsilon-power}
is the tensor product of these $\ZZ/d\ZZ$-fixed spaces: 
$$
\big( \epsilon^{-i} \otimes \CC[x_j] \otimes \wedge \CC{x_j} \big)^{\ZZ/d\ZZ}
=\begin{cases}
 \big( \CC[x_j^d] \otimes 1  \big) \oplus 
 \big( x_j^{d-1}\CC[x_j^d] \otimes x_j \big) & \text{ if }i=0,\\
\big( x_j^{d-i} \CC[x_j^d] \otimes 1  \big) \oplus 
 \big( x_j^{d-1-i} \CC[x_j^d] \otimes x_j \big) 
  & \text{ if }1 \leq i \leq d-1.
 \end{cases}
 $$
 
 \subsection*{The $i=0$ case}
 The space  $( \CC[x_j^d] \otimes 1  ) \oplus 
 ( x_j^{d-1}\CC[x_j^d] \otimes x_j )$ is the tensor product of a symmetric algebra $\CC[x_j^d]$
 on the generator $x_j^d$  with an exterior algebra on the generator ${x_j^{d-1} \otimes x_j}$.
Consequently,  the $(\ZZ/d\ZZ)^n$-fixed space~\eqref{relative-invariants-of-epsilon-power} is 
a symmetric algebra $\CC[x_1^d,\ldots,x_n^d]$ tensored with the exterior algebra on the generators
$\{ x_j^{d-1} \otimes x_j\}_{j=1}^n$.  Since the residual action of $\symm_n$ permutes the subscripts as
 usual, one concludes that
 the last expression in~\eqref{factor-rewriting} is $P_{\symm_n}(\chi^\lambda; q^d, tq^{d-1})$
  for $i=0$.

 \subsection*{ The $1 \leq i \leq d-1$ case}
We proceed as in the $i=0$ case.  First enlarge coefficients from 
 $\CC[x_i]$ to $\Frac(\CC[x_i])$.  Then  $( x_j^{d-i} \CC[x_j^d] \otimes 1  ) \oplus 
 ( x_j^{d-1-i} \CC[x_j^d] \otimes x_j )$
  is a free module of rank $1$ with basis element
 $x_j^{d-i}$ over the subalgebra of $\Frac(\CC[x_i]) \otimes \wedge  \CC{x_j}$
which is the tensor product of a symmetric algebra $\CC[x_j^d]$
 on generator $x_j^d$  with an exterior algebra on generator ${x_j^{-1} \otimes x_j}$.
 Consequently,  the $(\ZZ/d\ZZ)^n$-fixed space~\eqref{relative-invariants-of-epsilon-power} is 
 a free module of rank $1$ with basis element
 $(x_1 \cdots x_n)^{d-i}$ over the subalgebra of $\Frac(S) \otimes \wedge V^*$
given as the tensor product of the symmetric algebra $\CC[x_1^d,\ldots,x_n^d]$
  with the exterior algebra on the generators $\{x_j^{-1} \otimes x_j\}_{j=1}^n$.
  Since $\symm_n$ still permutes the subscripts, the last expression in \eqref{factor-rewriting} 
  is just $q^{n(d-i)} P_{\symm_n}(\chi^\lambda; q^d, tq^{-1})$.
 
 \vskip.1in
 We substitute these expressions for each $i$  in 
 \eqref{Frobenius-reciprocity-simplification} to obtain the product on the right in the proposition.
 \end{proof}

 \begin{remark}
   For the sake of the reader wishing to compare notation with that of Koike \cite{Koike}, note that the group we call $W_n=G(d,1,n)$ is his group $G_{n,d}$, and the group
   that we call $G(de,e,n)$ below is his group $G_{n,d,e}$.  Also, because he works
   with $S(V) \otimes \wedge(V)$ rather than $S(V^*) \otimes \wedge(V^*)$,
   his Theorem 1 calculates what we denote here by
   $P_{W_n}(\left( \cchi^{\underline{\lambda}} \right)^*;q,t)$, where
   one can identify the contragredient representation
   of $\cchi^{\underline{\lambda}}=\cchi^{(\lambda^{(0)}, \lambda^{(1)},\ldots,\lambda^{(r-1)})} $ 
   with
   $
   \left( \cchi^{\underline{\lambda}} \right)^*=\cchi^{(\lambda^{(0)}, \lambda^{(r-1)},\ldots,\lambda^{(2)},\lambda^{(1)})}.
   $
 \end{remark}

\begin{corollary}
\label{mainthmforG(d,1,n)}
The conclusion of Theorem~\ref{maintheorem} holds for the Shephard-Todd family $G(d,1,n)$ with $d \geq 2$.
\end{corollary}
\begin{proof}
  According to
  \eqref{Hilb-and-P-relation}, we should compute $P_W((\cchi_{\wedge^r V})^*;q,t)$
  for $W=G(d,1,n)$ acting on $V=\CC^n$ and $0 \leq r \leq n$.
One can check that $(\cchi_{\wedge^r V})^*$ is the $W$-irreducible character $\chi^{\underline{\lambda}}$
for $\underline{\lambda}=(\lambda^{(0)},
\ldots, \lambda^{(d-1)})$ where
$$
\begin{aligned}
\lambda^{(0)}=(n-r),\quad
 \lambda^{(1)}=\lambda^{(2)}=\cdots=\lambda^{(d-2)}=\varnothing,\quad
 \text{ and }\ \ 
 \lambda^{(d-1)}=(1^r).
 \end{aligned}
$$
Hence  Theorem~\ref{Kirillov-Pak-generalized} together with
Corollary~\ref{type-A-hook-content-formula} gives
$$
\begin{aligned}
P_{G(d,1,n)}(\left(\cchi_{\wedge^r V}\right)^*; q,t) 
&=P_{\symm_{n-r}}(\chi^{(n-r)}; q^d,tq^{d-1}) 
 \ q^{r}\ P_{\symm_r}(\chi^{(1^r)}; q^d,tq^{-1}) \\
&=  \prod_{ x \in (n-r) } \frac{1+tq^{d-1+dc(x)}}{1-q^{dh(x)}} 
\ q^{r} \ q^{d\binom{r}{2}}  \prod_{ x \in (1^r) } \frac{1+tq^{-1+dc(x)}}{1-q^{dh(x)}} \\
&= q^{r+d\binom{r}{2}}  \
  \prod_{ i=1}^{n-r} \frac{1+tq^{d-1+d(i-1)}}{1-q^{di}} 
  \ \prod_{i=1}^r \frac{1+tq^{-1-d(i-1)}}{1-q^{di}} \\
  &= q^{r+d\binom{r}{2}} \ \qbin{n}{r}{q^d}  
     \ \frac{(-tq^{d-1};q^d)_{n-r} \ (-tq^{-1};q^{-d})_n}{(q^d;q^d)_n}\\
\end{aligned}
$$
using \eqref{q-binomial-definition} and \eqref{principal-specialization-of-elementary}
in the second and third equalities.  As
$G(d,1,n)$ has rank $n$ and 
exponents $$(e_1,\ldots,e_{n})=(d-1,2d-1,\ldots,nd-1),$$
its smallest exponent is $e_1=d-1$ and the gap between exponents
is $a=d$.  Hence this last expression agrees with 
that in Theorem~\ref{maintheorem}${}^\prime$, which is equivalent to Theorem~\ref{maintheorem}.
\end{proof}

\subsection*{Koike's formula for $G(de,e,n)$}

Koike also generalized Theorem~\ref{Kirillov-Pak-generalized}
from the wreath product groups $G(d,1,n)$
to the entire infinite family $G(de,e,n)$ in Shephard and Todd's classification
of irreducible complex reflection groups \cite{ShephardTodd}.
Here $G(de,e,n)$ for positive integers $d,e,n$ is the
the kernel of the degree one character $\epsilon^d: G(de,1,n) \rightarrow \CC^\times$,
where $\epsilon$ was defined in \eqref{definition-of-epsilon-character}.
In other words, $G(de,e,n)$ is the group of monomial $n \times n$ matrices
whose nonzero entries are all $(de)^{th}$ roots-of-unity and
for which the product of the nonzero entries is a $d^{th}$
root-of-unity.
Although the group $G(de,e,n)$ is only coincidental when $e=1$, we find it worthwhile
to state his result and then use it to compute the Hilbert series
of $\left( S \otimes \wedge V^* \otimes \wedge^r V\right)^{G(de,e,n)}$, so that we can see how it would
differ from the form in Theorem~\ref{maintheorem} when $e \geq 2$.

First we recall the parametrization of irreducible $G(de,e,n)$-representations; see, e.g., \cite[\S2]{Koike}.
Given an irreducible $G(de,1,n)$-character $\cchi^{\underline{\lambda}}$
corresponding to a multipartition $\underline{\lambda}=(\lambda^{(0)},\lambda^{(1)},\ldots,\lambda^{(de-1)} )$ of $n$,
consider the superscripts $i$ in $\lambda^{(i)}$ as taken modulo $de$,
so that $\lambda^{(i+de)}=\lambda^{(i)}$ for all integers $i$.
Then the restriction of $\cchi^{\underline{\lambda}}$ from $G(de,1,n)$ to $G(de,e,n)$
depends only upon the orbit of $\underline{\lambda}$
under the operation that replaces $\lambda^{(i)}$ by $\lambda^{(i+d)}$ for
all $i$.  If one fixes a representative $\underline{\lambda}$ of this orbit
and defines a positive integer
$$
\mu(\underline{\lambda}):=\min\{m \geq 1: \lambda^{(i)}=\lambda^{(i+dm)} \text{ for all }i\},
$$
then $\cchi^{\underline{\lambda}}$ decomposes upon restriction to $G(de,e,n)$
into $e/\mu(\underline{\lambda})$ inequivalent $G(de,e,n)$-irreducible
characters, and each $G(de,e,n)$-irreducible arises once
in this way.  Koike's result may then be stated as follows.

\begin{theorem} \cite[Theorem 2]{Koike}
  \label{Koike-theorem}
  For a multipartition $\underline{\lambda}$ of $n$ parametrizing
  an irreducible character $\cchi^{\underline{\lambda}}$ of $G(de,1,n)$,
  let $\hat{\cchi}^{\underline{\lambda}}$ denote any $G(de,e,n)$-irreducible
  constituent of the restriction to $G(de,e,n)$.
  Then 
$$
P_{G(de,e,n)}(\hat{\cchi}^{\underline{\lambda}};q,t)
= \sum_{v=0}^{\mu(\underline{\lambda})-1}
P_{\symm_{n_{dv}}}(\chi^{\lambda^{(dv)}};q^{de},tq^{de-1})
  \cdot \prod_{i=1}^{de-1} q^{n_{dv+i} \cdot (de-i)}
    \cdot P_{\symm_{n_{i+dv}}}(\chi^{\lambda^{(i+dv)}};q^{de},tq^{-1}).
    $$
    In particular, the answer is
    independent of the chosen irreducible constituent $\hat{\cchi}^{\underline{\lambda}}$ 
    and depends only on $\chi^{\underline{\lambda}}$.
\end{theorem}

\noindent
Note that for $e=1$, every $\underline{\lambda}$
has $\mu(\underline{\lambda})=1$, so that the sum has only the $v=0$ term,
recovering Theorem~\ref{Kirillov-Pak-generalized}.

\begin{corollary}
\label{Koike-theorem-wedge-corollary}
  Let $W=G(de,e,n)$ with $d,e \geq 1$ and $n, de \geq 2$.
  Then the various Hilbert series 
$$
\Hilb((S\ot \wedge V^*\ot \wedge^r V)^W, q,t)=
P_W((\cchi_{\wedge^r V})^*; q,t)
$$
for $0 \leq r \leq n$ have these formulas:
\begin{itemize}
\item When $r=0$, it is
$$  
  \frac{(-tq^{de-1};q^{de})_{n-1} (1+tq^{dn-1})}
  {(q^{de};q^{de})_{n-1} (1-q^{dn})}\, .
  $$
    \item When $1 \leq r \leq n-1$ and $d \geq 2$, it is
$$  \qquad\quad
  q^{de\binom{r}{2}+r}\ 
   \frac{(-tq^{-1};q^{-de})_r (-tq^{de-1};q^{de})_{n-1-r}} 
  {(q^{de};q^{de})_r (q^{de};q^{de})_{n-r}(1-q^{dn})} 
  \left(
    1-q^{den}+tq^{-1}\big(q^{de(n-r)}(1-q^{dn})+q^{dn}-q^{den}\big)
  \right) \, .
  $$
      \item When $1 \leq r \leq n-1$ and $d=1$, it is
$$
\begin{aligned}        \qquad\quad
  q^{e\binom{r}{2}}\
   \frac{(-tq^{-1};q^{-e})_{r-1} (-tq^{e-1};q^{e})_{n-1-r}} 
  {(q^{e};q^{e})_{r} (q^{e};q^{e})_{n-r} (1-q^n)}
  &\left( q^r (1+tq^{(n-r)e-1}) (1+tq^{-(r-1)e-1}) (1-q^n) \right. \\
  & \qquad + q^{(n-r)(e-1)-1} (1+tq^{e-1}) (1-q^n) (q+t) \\
  & \qquad \left. + q^{r-1} (1+tq^{-(r-1)e-1}) (q+t)(q^n-q^{n(e-1)}) \right) \, .
\end{aligned}
$$
\item When $r=n$, it is
  $$
  q^{de\binom{n-1}{2}+n-1} \
  \frac{(-tq^{-1};q^{-de})_{n-1} }
  {(q^{de};q^{de})_{n-1} (1-q^{dn})} \cdot
  \begin{cases} (q^{de(n-1)+1}+t) & \text{ if }d \geq 2,\\
    (q^{(n-1)(e-1)}+t) & \text{ if }d=1\ .
  \end{cases}
  $$
\end{itemize}
\end{corollary}
\begin{proof}
  We only sketch the somewhat tedious proofs, which are of the same nature as
  those for Corollary~\ref{mainthmforG(d,1,n)}.
  
  The $G(de,e,n)$-representation $\left(\wedge^r V\right)^*$ is
  the restriction of the $G(de,1,n)$-representation $\left(\wedge^r V\right)^*$,
  and the latter has character $\cchi^{\underline{\lambda}}$ where
  $\underline{\lambda}=((n-r),\varnothing, \varnothing,\ldots,\varnothing,(1^r))$.
  Thus one can apply Theorem~\ref{Koike-theorem} to compute
  $P_{G(de,e,n)}(\hat{\cchi}^{\underline{\lambda}}; q,t)$.  In this case,
  $\mu(\underline{\lambda})=e$, and so the sum in Theorem~\ref{Koike-theorem} always has $e$ terms.
  
  When $r=0$, since $\underline{\lambda}$ has only one nonempty component $\lambda^{(0)}=(n)$,
  each summand in Theorem~\ref{Koike-theorem} for $v=0,1,2,\ldots,e-1$
  has only one non-unit factor:  the $v=0$ summand is the factor
  $P_{\symm_n}(\cchi^{(n)};q^{de},tq^{de-1})$, and each of the $v=1,2,\ldots,e-1$ summands
  is the factor indexed by $i=d(e-v)$ in the product.
  After pulling out common factors, one can sum the geometric series over $v=1,2,\ldots,e-1$,
  and some simplification then leads to the formula given in the corollary.
  The proof for $r=n$ is extremely similar.

  When $1 \leq r \leq n-1$, since $\underline{\lambda}$ has two nonempty components $\lambda^{(0)}=(n-r)$
  and $\lambda^{(de-1)}=(1^r)$, each summand for Theorem~\ref{Koike-theorem} for $v=0,1,2,\ldots,e-1$
  is a product of two non-unit factors.
  \begin{itemize}
    \item The $v=0$ summand is the factor in front of the product
      times the $i=1$ factor within the product.
    \item If $d=1$, the $v=1$ summand is similarly the factor in front times the
  factor indexed by $i=e-1$ within the product, while each
  of the $v=2,3,\ldots,e-1$ summands is the product of the two factors for $i=e-v, e-v+1$
  in the product.
\item If $d \geq 2$ so that $e-1 < de-1$, each of the $v=1,2,\ldots,e-1$ summands
  is the product of the two factors for $i=d(e-v), d(e-v)+1$ in the product.
\end{itemize}
In each case, one pulls out common factors, sums a geometric series,
and simplifies to obtain the formula.  
\end{proof}

\begin{remark}
  \label{G(de,e,n)-answer-comparison-remark}
One can compare the formulas for the Hilbert series in Corollary~\ref{Koike-theorem-wedge-corollary}
to see how they differ from the formulas in Theorem~\ref{maintheorem}.  Note that $G(de,e,n)$ has
\begin{equation}
\label{G(de,e,n)-exponents-coexponents}
\begin{aligned}
(d_1,d_2,\ldots,d_n)&= (de,2de,\ldots,(n-1)de,dn)\, ,\\
(e_1,e_2,\ldots,e_n)&= (de-1,2de-1,\ldots,(n-1)de-1,dn-1)\, ,\\
(e^*_1,e^*_2,\ldots,e^*_n)&=
 \begin{cases}
   (1,1+de,1+2de,\ldots,1+(n-2)de,1+(n-1)de) & \text{ if }d \geq 2,\\
   (1,1+e,1+2e,\ldots,1+(n-2)e,(n-1)(e-1)) & \text{ if }d=1\, .\\
 \end{cases}
 \end{aligned}
\end{equation}
It is then not hard to check that the cases $r=0$ and $r=n$ of
Corollary~\ref{Koike-theorem-wedge-corollary} agree with the formula given in
Theorem~\ref{maintheorem}, as predicted by Solomon's formula \eqref{Solomon-consequence} for $r=0$
and Theorem~\ref{Shepler-product} for $r=n$.

For $1 \leq r \leq n-1$, 
one can check using~\eqref{G(de,e,n)-exponents-coexponents} that the formula in Theorem~\ref{maintheorem} would assert for $d \geq 2$
that 
$$
\Hilb((S\ot \wedge V^*\ot \wedge^r V)^W, q,t)=
  q^{de\binom{r}{2}+r} \frac{(-tq^{-1};q^{-de})_r (-tq^{de-1};q^{de})_{n-1-r}(1+tq^{d \min\{(n-r)e,n\}-1})} 
  {(q^{de};q^{de})_r (q^{de};q^{de})_{n-r}}\, .
  $$
In general, this formula is {\it incorrect}---one can check that it disagrees with the corresponding
  expression in Corollary~\ref{Koike-theorem-wedge-corollary} except when $e=1$,
  namely when $G(de,e,n)=G(d,1,n)$, a coincidental group.

  Similarly, for $d=1$, so that  $W=G(e,e,n)$, one can check using~\eqref{G(de,e,n)-exponents-coexponents}  that the formula in Theorem~\ref{maintheorem} would assert that
$$
\begin{aligned}
  & \Hilb((S\ot \wedge V^*\ot \wedge^r V)^W, q,t)\\
  &=
  q^{e\binom{r}{2}+r} \ \cdot\ 
  \frac{(-tq^{-1};q^{-e})_{r-1} (-tq^{e-1};q^{e})_{n-1-r}}
         {(q^{e};q^{e})_{r-1} (q^{e};q^{e})_{n-r}} \\
         & \qquad \qquad \qquad
         \cdot\ \frac{ (1+tq^{\min\{(n-r)e,n\}-1})
                  (1+tq^{-\min\{1+(r-1)e,(n-1)(e-1)\}})}
                {(1-q^{re})(1-q^n)} \cdot
    \left(
     1-q^{ne-re}+q^{ne-re-n}-q^{ne-n}
   \right)\, .
  \end{aligned}
  $$
  In general, this formula is also {\it incorrect}---one 
  can check it disagrees with the $d=1$ case of Corollary~\ref{Koike-theorem-wedge-corollary} except
  when $r=1$, in which case both formulas agree with  \eqref{Reiner-Shepler}
  since $G(e,e,n)$ is a duality group.
\end{remark}

\section{Conjectured explicit basis}
\label{conjecture}
Here we strengthen Theorem~\ref{maintheorem}
to conjecture an explicit basis of the space
of invariants $(\SSS \otimes \wedge V^* \otimes \wedge V)^W$,
extending ideas of~\cite{Shepler05}. In fact, for all of the remaining coincidental groups $W$ (not type $A$ nor $G(d,1,n)$), we verify this
stronger conjecture in later sections to complete the verification of Theorem~\ref{maintheorem}.
Let us view
$$
M_r :=(S(V^*) \otimes \wedge V^* \otimes \wedge^r V)^W
$$
as a module over the exterior algebra
$(S(V^*) \otimes \wedge V^*)^W = \wedge_{S(V^*)^W}\{ df_1,\ldots,df_n\}$
via multiplication into the first two tensor factors.  
We require notation for letting a derivation act as an operator on $M_r$ by taking
partial derivatives of polynomial coefficients:
Given any derivation
$\theta =\sum_{j=1}^n h_j(x_1,\ldots, x_n) \otimes y_j$ in $S(V^*) \otimes V$, 
let
$$
\tilde{\theta}:S(V^*)\ot \wedge^k V^*\ot \wedge V\longrightarrow
S(V^*)\ot \wedge^{k+1} V^*\ot \wedge V
$$ 
be the differential operator defined by
$$
f \otimes \eta \otimes \eta'
\qquad{\longmapsto} \qquad
\sum_{j=1}^n \
\overline{h}_j\left( \tfrac{\del}{ \del x_1},\ldots, \tfrac{\del}{\del x_n} \right)
(f) \otimes x_j \wedge \eta \otimes \eta',
$$
with bar indicating that complex coefficients are
conjugated.
Here, each $x_i^m$ in the polynomial $h_j$ is replaced by the iterated partial derivative $\del^m/\del x_i^m$
and the resulting operator is applied to $f$.

Note that in the case where $\theta$ is the Euler derivation $\theta_E=\sum_{j=1}^n x_i \otimes y_j$
in $S(V^*)\ot V$,
the operator $\tilde{\theta}_E$ restricts to the
the usual exterior derivative $d: S(V^*)\mapsto S(V^*)\ot V^*$.
Restricting $\tilde{\theta}_E$ merely to a map $S(V^*) \otimes \wedge^0 V^* \otimes V \rightarrow S \otimes \wedge^1 V^* \otimes V$
gives the operator $\psi \mapsto d \psi$ from~\cite{ReinerShepler}.  We can now state our conjecture.

\begin{conjecture}
\label{structural-Molchanov}
For $W$ a coincidental reflection group, one may choose basic invarints $f_1,\ldots,f_n$ and 
basic derivations $\theta_1,\ldots,\theta_n$ so that each
$M_r=(S(V^*) \otimes \wedge V^* \otimes \wedge^r V)^W$ for $r=1, \ldots, n$ 
is a free module 
over the exterior algebra 
$R_r:=\wedge_{S(V^*)^W} \{df_1,\ldots,df_{n-r}\}$ with basis
$$
\tilde{\theta}_{i_1} \cdots \tilde{\theta}_{i_m} \left( \theta_{j_1} \wedge \cdots \wedge \theta_{j_r} \right) 
$$
for $0 \leq m \leq r$ and $1 \leq i_1 < \cdots < i_m \leq r$ and $1 \leq j_1 < \cdots <  j_r \leq n$.
For $r=0$, take basis element $1$.
\end{conjecture}

Equivalently one can write the conjectured basis more directly in terms of subsets of $\{1, \ldots, n\}$.
Define
$$
\begin{aligned}
  \tilde{\theta}_I&:=\tilde{\theta}_{i_1} \cdots \tilde{\theta}_{i_m}:
  { S(V^*)\ot\wedge^k V^*\ot\wedge V
  \longrightarrow S(V^*)\ot\wedge^{k+m} V^*\ot\wedge V}
  &&\text{ for }
       I=\{i_1 < \cdots < i_m\},\\
  \theta_J&:=\theta_{j_1} \wedge \cdots \wedge \theta_{j_r}
  \ \quad\text{ in } S(V^*)\ot \wedge^0\ot\wedge^r V
  &&\text{ for } J=\{j_1 < \cdots < j_r\},
\\
df_L&:=df_{\ell_1} \wedge \cdots \wedge df_{\ell_{m'}}
\text{ in } S(V^*)\ot\wedge^{m'}V^*\ot \wedge^0 V
&&\text{ for } L=\{\ell_1 < \cdots < \ell_{m'}\}
\, .
\end{aligned}
$$
Conjecture~\ref{structural-Molchanov} is then equivalent to the following statement.
\vskip.1in
\noindent
{\bf Conjecture~\ref{structural-Molchanov}${}^\prime$}.
{\it
  For $W$ a coincidental reflection group, one may choose
  basic invariants $f_1, \ldots, f_n$ and basic derivations
  $\theta_1,\ldots, \theta_n$
so that 
\begin{equation}
\label{definition-of-Mrk}
M_{r,k}:=(S(V^*) \otimes \wedge^k V^* \otimes \wedge^r V)^W
\quad\text{for}\quad
1\leq r,k \leq n
\end{equation}
is a free module over $S(V^*)^W$ with basis
$$
\{ df_L \cdot \tilde{\theta}_I \left( \theta_J \right) \}
$$
as one runs through all triples $(I,J,L)$ of
subsets $J \subset [n]$, $I \subset [r]$, $L \subset [n-r]$ 
with $|J|=r$ and $|I|+|L|=k$.
}

\begin{remark}
\label{known-special-cases-remark}
Note that several special cases of the equivalent Conjectures~\ref{structural-Molchanov} and ~\ref{structural-Molchanov}${}^\prime$ hold more generally:
\begin{itemize}
 \item Conjecture~\ref{structural-Molchanov} at $r=0$ holds for all reflection groups by~\cite{Solomon}.
\item Conjecture~\ref{structural-Molchanov} at $r=1$ holds for all duality groups 
by~\cite[Thm~1.1]{ReinerShepler}.
\item  Conjecture~\ref{structural-Molchanov}${}^\prime$ at $k=0$ 
holds for all reflection groups by~\cite[Thm.~3.1]{OrlikSolomon}.
\end{itemize}

\end{remark}

\begin{proposition}
\label{conjecture-proves-thm}
Conjecture~\ref{structural-Molchanov}${}^\prime$  (or its equivalent Conjecture~\ref{structural-Molchanov}) implies Theorem~\ref{maintheorem}. 
\end{proposition}
 \begin{proof}
The conjecture gives an $S^W$-basis for 
$(\SSS \otimes \wedge V^* \otimes \wedge^r V)^W$ where we again appreviate $S=S(V^*)$; we compare the polynomial degrees of basis elements in the conjecture with
the coefficient of $s^r$ on the right in Theorem~\ref{maintheorem}.
Since $\Hilb(S^W,q)=\prod_{i=1}^n (1-q^{d_i})^{-1}$, after clearing the denominator in Theorem~\ref{maintheorem}, 
we have
\begin{equation}
\left( \prod_{\ell=1}^{n-r}(1+q^{e_\ell}t) \right)
 \left(  \prod_{i=1}^{r}(1+q^{-e_i^*}t) \right) \sigma_r(q^{e_1*},\ldots,q^{e_n^*}).
\end{equation}
We check that this matches the $(q,t)$-bidegrees of the $S^W$-basis elements 
$
\{ df_L \cdot \tilde{\theta}_I \left( \theta_J \right) \}
$
in Conjecture~\ref{structural-Molchanov}${}^\prime$:
$$
\begin{aligned}
&\left(\prod_{\ell=1}^{n-r}(1+q^{e_\ell}t) \right)
 \left(  \prod_{i=1}^{r}(1+q^{-e_i^*}t) \right) \sigma_r(q^{e_1*},\ldots,q^{e_n^*})\\
 &=
 \left(
 \sum_{L \subset [n-r]} t^{|L|} q^{\sum_{\ell \in L} e_\ell} 
\right)
\left(
\sum_{I \subset [r]} t^{|I|} q^{- \sum_{i \in I} e_i^*}
\right)
\left(
\sum_{\substack{J \subset [n]:\\|J|=r}} q^{\sum_{j \in J} e_j^*} 
\right)
=\sum_{(I,J,L)} t^{|I|+|L|} q^{\sum_{\ell \in L} e_\ell + \sum_{j \in J}e_j^* - \sum_{i \in I} e_i^* }, 
\end{aligned}
$$
where the last sum runs through $(I,J,L)$ satisfying the conditions in   Conjecture~\ref{structural-Molchanov}${}^\prime$. Theorem~\ref{maintheorem} then follows 
since $df_L \cdot \tilde{\theta}_I \left( \theta_J \right)$ has $\wedge V^*$-degree $|I|+|L|$ and 
$S$-degree $\sum_{\ell \in L} e_\ell + \sum_{j \in J}e_j^* - \sum_{i \in I} e_i^*$.
\end{proof}

\section{The elements in Conjecture~\ref{structural-Molchanov} are invariant}
\label{invariance-of-operators-section}
As a precursor to verifying Conjecture~\ref{structural-Molchanov} or ~\ref{structural-Molchanov}${}^\prime$,
we check that the forms indicated there are indeed invariant under the action
of any reflection group $W$.

\begin{lemma}
\label{operatorinvariance}
For $V=\CC^n$, let $W\subset\text{GL}(V)$ be a group of isometries
and let $S=S(V^*)$.
\begin{itemize}
\item[(a)] The map
  $
  (S\ot V \ot 1) \times (S\ot\wedge^kV^*\ot \wedge^r V)
  \rightarrow
  S\ot \wedge^{k+1}V^* \ot \wedge^r V$,
  \
  $(\theta, \omega) \mapsto \widetilde{\theta}(\omega)$,
 is $W$-equivariant.
 \item[(b)]
 \rule{0ex}{3ex}
   For $\theta$ in $(S\ot 1 \ot V)^W$
   and $\om$ in $(S\ot\wedge^kV^*\ot \wedge^r V)^W$,
   the form $\tilde{\theta}(\omega)$ lies in
   $(S\ot\wedge^{k+1}V^*\ot \wedge^r V)^W$.
      \end{itemize}
   \end{lemma}
\begin{proof}
For any polynomials $h$, $f$ in $S$, let
$\del(h)(f):=
\bar{h}\big(\tfrac{\del}{\del x_1}, \ldots, \tfrac{\del}{\del x_n}\big)(f).$
The map 
$S\times S \rightarrow S$ given by
$(h, f)\mapsto \del(h)(f)$
is $W$-equivariant (see~\cite[\S 6]{Shepler05}):
For $g$ in $W$,
$$
g\big(\del(h)f\big) = \del(g h)(gf)\, .
$$
Now consider some 
$\theta = h \ot 1 \ot y_j$ in 
$S\ot 1 \ot V$ and 
$\omega = f \ot \eta \ot \eta'$ in $S\ot \wedge V^*\ot \wedge V$.
For any $g$ in $W$, $(g\theta, g\omega)$ maps to $g(\tilde{\theta}({\omega}))$ 
under the function in (a) since
$$
\begin{aligned}
g\big( \tilde{\theta}( \omega) \big)
&=
g\big( \del(h)f \ot x_j \wedge \eta \ot \eta' \big)
=g\big(\del(h) f\big) \ot g x_j \wedge g\eta \ot g \eta'
\\
&=\del\big(g(h)\big) (gf) \ot g(x_j \wedge  \eta) \ot g \eta'
=\widetilde{g\theta}\big( g( f\ot \eta \ot \eta') \big)
=\widetilde{g\theta}(g\omega)\, .
\end{aligned}
$$
Thus $(\theta, g\omega)$ maps to $g\big(\tilde{\theta}(\omega)\big)$ whenever $\theta$ is $W$-invariant,
and $\tilde{\theta}(\omega)$ is $W$-invariant whenever $\theta$ and $\omega$ are.
\end{proof}

\begin{remark}Recall that for any linear transformation $g$ in $\text{GL}(V)$,
the matrix recording the dual action of $g$ on $V^*$ with respect to a dual basis is just the inverse transpose of
the matrix recording the action of $g$ on $V$.
As $W$ is finite, we may and have assumed the matrices giving the action of $W$ are all unitary, so that
the matrix of $g$ in $W$ acting on $V^*$ is the complex conjugate (entry-wise) of the matrix of $g$ acting on $V$.
This explains why we
take the complex conjugates of coefficients in defining the operators $\tilde{\theta}$.
\end{remark}

\section{Conjecture~\ref{structural-Molchanov}
 agrees with the Gutkin-Opdam calculation}
\label{Gutkin-Opdam-section}

As a second precursor to verifying Conjecture~\ref{structural-Molchanov} (or ~\ref{structural-Molchanov}${}^\prime$), we check that it
correctly predicts the sum of the degrees of the homogeneous elements in a basis
for $(S \otimes \wedge^k V^* \ot \wedge^r V)^W$ over $S(V^*)^G$ when $W$ is a duality group.  
For any reflection group $W$, set (e.g., see~\cite[\S4.5.5, Remark~4.48]{Broue}))
\begin{equation}
\label{N-and-N-star-definitions}
\begin{aligned}
N&:=\#\{\text{reflections in }W\}={\textstyle\sum_{i=1}^n} e_i,\\
N^*&:=\#\{\text{reflecting hyperplanes for }W\}
=\textstyle\sum_{i=1}^n e^*_i\, .
\end{aligned}
\end{equation}
Recall that for any $W$-representation $U$, the
sum of the $U$-exponents 
$$\psi(U)=e_1(U)+\ldots + e_{\dim U}(U)\, $$
is the sum of the polynomial degrees of homogeneous elements
in any basis of $(S(V^*) \otimes U)^W$ as a free module over
$S(V^*)^W$.
The Gutkin-Opdam Lemma~\cite[Prop.~4.3.3, eqn.~(4.6)]{Broue} 
allows one to predict $\psi(U)$
as follows:
$$
\begin{aligned}
\psi(U)&=\sum_{H} \sum_{j=0}^{e_H-1} j \cdot \langle U \downarrow^W_{W_H}, \det{}^j \rangle_{W_H} \\
&=\sum_{H} D_H( U \downarrow^W_{W_H} ),
\end{aligned}
$$
where $H$ runs through the reflecting hyperplanes for $W$, with $W_H\subset W$ the pointwise stabilizer subgroup of $H$
and $e_H:=|W_H|$.  Here $D_H$ is the linear functional on the Grothendieck group $G_0(W_H)$
of $W_H$-representations that sends the $1$-dimensional representation $\det^j$ to $j$ for $j=0,1,\ldots,e_H-1$.

\begin{lemma}\label{Gutkin-Opdam-Calculation}
Let $W$ be a reflection group on $V$ and 
$U=\wedge^k V^* \otimes \wedge^r V$.
Then $$\psi(U)=
\binom{n-1}{k-1}\ \binom{n-1}{r}\ N + 
\binom{n-1}{k}\ \binom{n-1}{r-1}\ N^*\, .
$$
\end{lemma}
\begin{proof}
  Let $H$ be a reflecting hyperplane of $W$.
  The restrictions of $V^*$ or $V$ to $W_H$ each
  contain $n-1$ copies of the trivial $W_H$-representation
  as direct summands,
  and then one extra summand
  carrying the $1$-dimensional representations
  ${\det}^{e_H-1}$ or $\det$, respectively.
  We sum over all $k,r$, apply the Gutkin-Opdam Lemma,
and then extend
$D_H$ to be linear over $\CC[s,t]$
keeping~(\ref{N-and-N-star-definitions}) in mind:
$$
\begin{aligned}
\sum_{k,r} \psi(\wedge^k V^* \otimes \wedge^r V) \,\, t^k s^r
&= \sum_{H} \sum_{k,r} D_H( \wedge^k V^* \otimes \wedge^r V \downarrow^W_{W_H} ) \,\, t^k s^r \\
&= \sum_{H} D_H\left( (1+t)^{n-1} (1+t \det{}^{e_H-1}) \, (1+s)^{n-1} (1+s \det) \right) \\
&=(1+t)^{n-1} (1+s)^{n-1} \sum_{H} D_H \left( (1+t \det{}^{e_H-1})(1+s \det) \right) \\
&=(1+t)^{n-1} (1+s)^{n-1} \sum_{H} D_H \left( (1+t \det{}^{e_H-1}+s \det +st) \right) \\
&=(1+t)^{n-1} (1+s)^{n-1} \sum_{H} (t (e_H-1)+s) \\
&=(1+t)^{n-1} (1+s)^{n-1} (tN+sN^*)\, .
\end{aligned}
$$
The result then follows from
extracting the coefficient of $t^ks^r$.
\end{proof}

\begin{proposition}
\label{Gutkin-Opdam-matches}
Either of Theorem~\ref{maintheorem}
or Conjecture~\ref{structural-Molchanov}
correctly predicts 
$\psi(U)$ for each $U=\wedge^k V^* \otimes \wedge^r V$
with $0\leq k,r\leq n$
for any irreducible duality reflection group $W$.

\end{proposition}
\begin{proof}
  [Proof of Proposition~\ref{Gutkin-Opdam-matches}]
  On one hand, 
  $\sum_{k,r}\psi(U)t^ks^r=(1+t)^{n-1} (1+s)^{n-1} (tN+sN^*)$
  by Lemma~\ref{Gutkin-Opdam-Calculation}.
  We fix $r=0,1,\ldots,n$ and extract the coefficient of
  $s^r$ to obtain
$$
\sum_{k} \psi(\wedge^k V^* \otimes \wedge^r V) \,\, t^k 
= (1+t)^{n-1} \left( t N \binom{n-1}{r} + N^* \binom{n-1}{r-1} \right).
$$
On the other hand,  Theorem~\ref{maintheorem}  predicts
for each $r=0,\ldots, n$ that
$$
\begin{aligned}
\sum_{k} \psi(\wedge^k V^* \otimes \wedge^r V) \,\, t^k 
&= \lim_{q \rightarrow 1} \left[ \tfrac{\del}{\del q} 
       \left( \sigma_r(q^{e^*_1},\ldots,q^{e^*_n}) \cdot
          \prod_{i=1}^{n-r} (1+q^{e_i}t)  \cdot  \prod_{i=1}^{r} (1+q^{-e^*_i}t) \right) \right] \\
&= \sum_{\substack{J \subset \{1,2,\ldots,n\}: \\ |J|=r}} 
        \lim_{q \rightarrow 1} \left[ \tfrac{\del}{\del q} 
        \left( q^{\sum_{j \in J} e_j^*} \cdot
          \prod_{i=1}^{n-r} (1+q^{e_i}t)  \cdot  \prod_{i=1}^{r} (1+q^{-e^*_i}t) \right) \right] \\
&= \sum_{\substack{J \subset \{1,2,\ldots,n\}: \\ | J|=r}} 
    \left[
        (1+t)^n \sum_{j \in J} e^*_j \quad +\quad (1+t)^{n-1} \sum_{i=1}^{n-r} te_i \quad +\quad  (1+t)^{n-1} \sum_{i=1}^{r} (-te^*_i) 
    \right]  \\
&= (1+t)^{n-1}
      \left[
        \sum_{i=1}^n e^*_i (1+t) \binom{n-1}{r-1} + t \binom{n}{r} \left( \sum_{i=1}^{n-r} e_i - \sum_{i=1}^{r} e^*_i \right) 
    \right]\, .
\end{aligned}
$$
It only remains to check the bracketed expression is
$t N \binom{n-1}{r} + N^* \binom{n-1}{r-1}$.  We use
\eqref{N-and-N-star-definitions}:
$$
\begin{aligned}
\sum_{i=1}^n e^*_i (1+t)& \binom{n-1}{r-1} +  t \binom{n}{r} \left( \sum_{i=1}^{n-r} e_i - \sum_{i=1}^{r} e^*_i \right)\\
&=N^* (1+t) \binom{n-1}{r-1} + t \binom{n}{r} \left( N-rh \right) \\
 &= t \left( \binom{n-1}{r-1} N^* + \binom{n}{r} (N-rh) \right) 
          + N^*  \binom{n-1}{r-1}\\
 &= t \left( \binom{n-1}{r-1} N^* + \binom{n-1}{r-1}(N-rh)+ \binom{n-1}{r} (N-rh) \right) 
          + N^*  \binom{n-1}{r-1} \\
 &= t N \binom{n-1}{r} 
          + N^*  \binom{n-1}{r-1}\, .
\end{aligned}
$$
Here we used the fact that $N+N^*=nh$ (as $W$ is a duality group)
and $\binom{n-1}{r-1}(n-r) = \binom{n-1}{r}r$.
\end{proof}

\subsection*{Independence over the fraction field}
\label{independence}
We now explain  why the above Opdam-Gutkin calculation implies that Conjecture~\ref{structural-Molchanov} may be shown 
with an independence argument.

\begin{lemma} \cite[Lemma~4.1]{ReinerShepler}
\label{graded-commutative-algebra-lemma}
Let $A$ be a graded $k$-algebra and integral domain, and
$M\cong A^p$ a free graded $A$-module whose homogeneous basis
elements have degrees suming to $\psi$.
Then another set of homogeneous elements $\{n_1, \cdots, n_p\}$ in $M$ 
with the same degree sum $\sum_{i=1}^p \deg(n_i)=\psi$
form an $A$-basis for $M$ if and only if they are
linearly independent over the fraction field $K = \text{Frac}(A)$.
\end{lemma}

Thus one may verify Conjecture~\ref{structural-Molchanov}
by proving that the
basis elements of the module $(S(V^*) \otimes \wedge^k V^* \otimes \wedge^r V)^W$  
it predicts are linearly independent over the fraction field of $(S(V^*)$ for each $k, r$:
\begin{prop}
\label{just-check-independence}
Let $W$ be a coincidental reflection group
with any set of basic invariants $f_1,\ldots,f_n$ and basic derivations
$\theta_1,\ldots,\theta_n$.  Fix some $1\leq r \leq n$.
If the derviation differential forms
$$
\tilde{\theta}_{i_1} \cdots \tilde{\theta}_{i_k} \big( \theta_{j_1} \wedge \cdots \wedge \theta_{j_r} \big) 
\ \ \in S(V^*) \otimes \wedge^k V^* \otimes \wedge^r V 
$$
for $0 \leq k \leq r$ and $1 \leq i_1 < \cdots < i_k \leq r$ and $1 \leq j_1 < \cdots <  j_r \leq n$
are linearly independent over
the fraction field $\text{Frac}(S(V^*))$,
then they form a basis for 
$M_{r}=(S(V^*) \otimes \wedge V^* \otimes \wedge^r V)^W$ over the exterior subalgebra 
$R_r:=\wedge_{S(V^*)^W} \{df_1,\ldots,df_{n-r}\}$. 
In this case, $M_{r}$ is a free module over $R_r$.
\end{prop}
\begin{proof}
Use Proposition~\ref{Gutkin-Opdam-matches} together with Lemma~\ref{graded-commutative-algebra-lemma}
 and Lemma~\ref{operatorinvariance}.
\end{proof}
By Proposition~\ref{just-check-independence}, it suffices to check various determinants are nonzero to prove
Conjecture~\ref{structural-Molchanov}.

\section{The main result and Molchanov's hypothesis}
\label{mainproof}
In this section, we outline the proof of our main result, Theorem~\ref{maintheorem}, from the Introduction:
\vskip.1in
\noindent
{\bf Theorem~\ref{maintheorem}.}
For any coincidental complex reflection group $W$ acting on $V=\CC^n$, 
$$
\Hilb\left((S(V^*)\otimes \wedge V^* \otimes \wedge V)^W,q,t ,s \right) 
= \sum_{r=0}^n \sigma_r(q^{e_1^*}, \ldots, q^{e_n^*})
   \frac{
  \prod_{i=1}^{r} (1+q^{e^*_i}t)\
   \prod_{i=1}^{n-r} (1+q^{e_i}t) s^r }
      { \prod_{i=1}^{n} (1- q^{d_i})  }\, .
$$
\begin{proof}

We proceed in essentially three cases, some of which prove the 
stronger Conjecture~\ref{structural-Molchanov}${}^\prime$; see Lemma~\ref{conjecture-proves-thm}.

\begin{itemize}
\item
For the Weyl groups of type $A_n$ and the
Shephard-Todd family $G(d,1,n)$, Theorem~\ref{maintheorem} was deduced in
Section~\ref{G(d,1,n)-section}, as 
Corollaries~\ref{type-A-hook-content-formula} and ~\ref{mainthmforG(d,1,n)}.

\item
We prove Conjecture~\ref{structural-Molchanov}${}^\prime$ uniformly in Section~\ref{two-dimensional-section}
for all coincidental groups of rank $2$, relying on Proposition~\ref{just-check-independence} so as to only 
check that certain determinants are nonzero;  see Theorem~\ref{main-thm-dim-2}.

\item
This leaves the exceptional real type $H_3$ and Shephard groups $G_{25}, G_{26}, G_{32}$ of ranks $3$ or $4$, where we checked 
Conjecture~\ref{structural-Molchanov}${}^\prime$ in {\tt Mathematica} via Proposition~\ref{just-check-independence},
for these choices of $\{f_i\}, \{\theta_i\}$:
\begin{itemize}
\item For $H_3$, we used $\{f_i\}$ from
Saito, Yano, and Sekiguchi~\cite{SaitoYanoSekiguchi},
with $\theta_i=\sum_{j=1}^n \frac{\partial f_i}{\partial x_j} \otimes y_j$.  
\item For $G_{25}, G_{26}, G_{32}$, we used 
$\{f_i\}$ and $\{\theta_i\}$ from Orlik and Terao~\cite[Appendix~B.3]{OrlikTerao}. \qedhere
\end{itemize}
\end{itemize}
\end{proof}

\begin{remark}
Theorem~\ref{maintheorem} is very closely related to a hypothesis stated by
Molchanov
~\cite[\S 8]{Molchanov}.
Molchanov's formulation\footnote{Warning: translation of his paper to English
  introduced two typos---instead of $\prod_{i=1}^{r}$ and $\prod_{i=1}^{n}$, it has $\prod_{i=1}^{m}$ for both.}
differs from Theorem~\ref{maintheorem} both
in that he assumes that $W$ is a {\it real} reflection group
(so that as $W$-representations, $V^*\cong V$ and $e^*_i=e_i$)
and in that he assumes that the exponents are {\it distinct}.
Unfortunately, the scope of his hypothesis is off,
and Theorem~\ref{maintheorem} seems to be the correct
formulation.  In fact, combining
Remark~\ref{G(de,e,n)-answer-comparison-remark}
with the data on the noncoincidental exceptional reflection groups presented 
in Section~\ref{data-section},
one sees that the formula in Theorem~\ref{maintheorem}
holds for an irreducible reflection group $W$, real or complex,
{\it if and only $W$ is coincidental}.
\end{remark}

\section{Reflection groups in two dimensions}
\label{two-dimensional-section}
Here we verify Conjecture~\ref{structural-Molchanov}${}^\prime$ for rank $2$
coincidental reflection groups using Proposition~\ref{just-check-independence}.
In two dimensions, the coincidental groups $W$ are the same as the irreducible duality groups.
Recall as notaion that $W$ acts on $V=\CC^2$ with $\CC$-basis $\{y_1,y_2\}$ and on
$V^*$ with dual  $\CC$-basis $\{x_1,x_2\}$, so that $S(V^*)=\CC[x_1,x_2]$.

\begin{lemma}\label{therecanbeonlyone}
For $W$ any rank $2$ irreducible finite reflection group acting on $V=\CC^2$,
the $W$-invariant space
$\left( V^* \ot V^* \ot \wedge^2 V\right)^W$ is $1$-dimensional over $\CC$,
spanned by 
\begin{equation}
\label{omega-two-up-to-rescaling}
(x_1 \ot x_2 -x_2 \ot x_1)\ot y_1\wedge  y_2.
\end{equation}
\end{lemma}
\begin{proof}
Again, we write $S=S(V^*)$.  As a $\text{GL}(V)$-representation, 
$V^* {\ot V^*} \cong S^2 \oplus \wedge^2 V^*$.
This isomorphism restricts one of $W$-representations, and hence
$$
\begin{aligned}
\left( V^* \ot V^* \ot \wedge^2 V\right)^W
\cong
\left( S^2 \ot \wedge^2 V \right)^W
  \oplus
     \left(\wedge^2 V^* \ot \wedge^2 V \right)^W\, .
\end{aligned}
$$
We analyze the two direct summands in this last expression.
Since $\wedge^2 V$ is $1$-dimensional spanned by $dy_1 \wedge dy_2$ and 
carries the $W$-character $\det$, for any $W$-representation $U$, the $W$-fixed space
$\left(U \ot \wedge^2 V \right)^W$ will be the tensor product of the $\det^{-1}$-isotypic component of $U$
with $\CC dy_1 \wedge dy_2$. 

For $U=\wedge^2 V^*$, this  $\det^{-1}$-isotypic component
is $1$-dimensional, spanned by $x_1 \otimes x_2 - x_2 \otimes x_1$.

For $U=S^2$, we argue that this $\det^{-1}$-isotypic component will vanish. 
For reflection groups $W$ of any rank, the  $\det^{-1}$-isotypic component of $S$ is the free
$S^W$-module $S^W \cdot Q$ of rank $1$ (by Stanley~\cite[Thm.~3.1]{Stanley},
see also~\cite[Ex.~6.40]{OrlikTerao})
for $Q$ the product of the linear
forms defining the reflecting hyperplanes with degree $N^*=e_1^*+\cdots+e_n^*$,
the number of reflecting hyperplanes. 
But $N^* > 2$ for any irreducible reflection group $W$ of
rank at least $2$ and hence $S^2$ has zero intersection with this $\det^{-1}$-isotypic component.

Thus
$\left( V^* \ot V^* \ot \wedge^2 V\right)^W
\cong \left(\wedge^2 V^* \ot \wedge^2 V \right)^W$ is spanned by 
$(x_1 \ot x_2 -x_2 \ot x_1)\ot y_1\wedge  y_2$.
\end{proof}

\begin{proposition}
\label{homogeneous-polynomial-fact}
If $h(x_1,\ldots,x_n)$ in $S(V^*)$ is homogeneous and nonzero, then
$\bar{h}(\tfrac{\partial}{\partial x_1},\ldots, \tfrac{\partial}{\partial x_n})(h) \neq 0$.
\end{proposition}
\begin{proof}
Express $h$ of degree $d$ as a finite sum of monomials,
$$
h=
\sum_{\substack{\alpha=(\alpha_1,\ldots,\alpha_n):\\\sum_i \alpha_i=d}}
   c_\alpha x_1^{\alpha_1} \cdots x_n^{\alpha_n}
=\sum_\alpha c_\alpha \xx^{\alpha},
$$
and write $(\del x_i)^m$ for $\del^{m}/ \del x_i^{m}$.
For $\sum_i \alpha_i =\sum_i \beta_i$, the monomial
$
(\del^{\alpha_1})^{\alpha_1} \cdots  (\del x_n)^{\alpha_n} (\xx^\beta) 
$ 
is
$
\alpha_1! \cdots \alpha_n!$
if $\beta=\alpha$ but vanishes for $\beta \neq \alpha$.  Thus by linearity, 
$$
\bar{h}(\partial x_1,\ldots, \partial x_n)(h) 
=\sum_{\alpha, \beta}\,  \bar{c}_\alpha \, c_\alpha \,\, (\partial x_1)^{\alpha_1} \cdots  (\partial x_n)^{\alpha_n} (\xx^\beta)
=\sum_{\alpha} \vert c \vert^2 \,\, \alpha_1! \cdots \alpha_n!\, ,
$$
which is strictly positive as long as at least one $c_\alpha \neq 0$ in $\CC$, 
that is, as long as $h \neq 0$ in $S$.
\end{proof}

\begin{prop}
\label{main-thm-dim-2}
Conjecture~\ref{structural-Molchanov} holds for any
irreducible duality group $W$ acting on $\CC^2$:
For $S=S(V^*)$ and
$M_{r}=(S \otimes \wedge V^* \otimes \wedge^r V)^W$,
there is a set of basic invariants $f_1$, $f_2$ and basic derivations
$\theta_1$, $\theta_2$ so that
\begin{itemize}
\item
$M_0$ is free over $R_0=\wedge_{S^W} \{df_1,df_{2}\}$ with basis $\{1\}$;
\item
$M_1$ is free over $R_1= {S^W} df_1$ with basis $\{\theta_1$, $\theta_2$, 
$\tilde{\theta}_{1}(\theta_1)$, $\tilde{\theta}_{1}(\theta_2)\}$;
\item
$M_2$ is free over $R_2={S^W}$ with basis 
$\{\theta_1\wedge \theta_2$,
$\tilde{\theta}_{1}(\theta_1\wedge \theta_2)$, $\tilde{\theta}_{2}(\theta_1\wedge\theta_2)$, 
$\tilde{\theta}_1 \tilde{\theta}_{2}(\theta_1\wedge\theta_2)\}$.
\end{itemize}
\end{prop}
\begin{proof}
Recall from Remark~\ref{known-special-cases-remark} that Conjecture~\ref{structural-Molchanov} is known
for all duality groups when $r\in\{0,1\}$ or $k=0$.  
Since $0 \leq r \leq n =2$, it hence suffices to only consider here the case $r=n=2$ and $k\in \{1,2\}$.

By Proposition~\ref{just-check-independence}, we need only check the linear independence 
over the fraction field $K$ of $S$ of 
$$
\tilde{\theta}_{1}(\theta_1\wedge \theta_2) \text{ and }\tilde{\theta}_{2}(\theta_1\wedge\theta_2)
\quad\text{in}\quad
M_{2,1}:=(S \otimes \wedge^1 V^* \otimes \wedge^2 V)^W
\quad
\text{(the case $k=1$)}
$$
and also check that
$$
\quad\quad\quad\quad\quad 0 \neq \tilde{\theta}_1 \tilde{\theta}_{2}(\theta_1\wedge\theta_2)
\quad\text{in}\quad
M_{2,2}:=(S \otimes \wedge^2 V^* \otimes \wedge^2 V)^W
\quad
\text{(the case $k=2$)}\, .
$$
Recall that 
$
\theta_1 \wedge \theta_2= Q \otimes 1 \otimes y_1 \wedge y_2
$
by~\cite[Thm.~3.1]{OrlikSolomon}
as $\wedge^2 V= \CC y_1 \wedge y_2$.
Again, we write $\del x_i$ for $\del/\del x_i$.

\subsection*{The case $n=r=2$ and $k=2$.}
The derivation differential form
$$
\tilde{\theta}_1 \tilde{\theta}_{2}(\theta_1\wedge\theta_2)
=
\tilde{\theta}_1 \tilde{\theta}_2( Q  \otimes 1 \otimes y_1 \wedge y_2)
= \bar{Q}(\del x_1,\del x_2) (Q) \otimes x_1 \wedge x_2 \otimes y_1 \wedge y_2
$$
is nonzero since Lemma~\ref{homogeneous-polynomial-fact} implies that
the scalar $\bar{Q}(\del x_1,\del x_2) (Q)$ 
in $\CC$ is nonzero.

\subsection*{The case $n=r=2$ and $k=1$.}
Here, one must check that 
\begin{equation}
\label{omegas-defined}
\begin{aligned}
\omega_1:= & \hphantom{x}\tilde{\theta}_1 (\theta_1 \wedge \theta_2) =\tilde{\theta}_1( Q  \otimes 1 \otimes y_1 \wedge y_2)
\quad\text{ and }\\
\omega_2:= & \hphantom{x}\tilde{\theta}_2 (\theta_1 \wedge \theta_2)= \tilde{\theta}_2( Q  \otimes 1 \otimes y_1 \wedge y_2)
\end{aligned}
\end{equation}
are $K$-linearly independent in 
$S \otimes V^* \otimes \wedge^2 V$.
Our strategy is to 
make convenient choices for basic derivations $\theta_1, \theta_2$
that identify $\omega_1, \omega_2$ more concretely.  
As explained in~\cite[Appendix~B.2]{OrlikTerao}, one may choose in rank $2$ 
\begin{equation}
\label{our-choice-of-rank-2-thetas}
\begin{array}{rcccl}
  \theta_1  &=& x_1 \otimes 1 \otimes y_1 &+&  x_2 \otimes 1 \otimes   y_2
  \quad\text{ and }\\
\rule{0ex}{3ex}
\theta_2  &=& -\frac{\partial Q}{\partial x_2} \otimes 1 \otimes y_1
                   &+&\frac{\partial Q}{\partial x_1}  \otimes 1 \otimes y_2\, .
\end{array}
\end{equation}
With this choice,
\begin{equation}
\label{omegas-more-explicitly}
\begin{aligned}
\omega_1
&=\left(
\tfrac{\partial Q}{\partial x_1}  \otimes x_1
+\tfrac{\partial Q}{\partial x_2}  \otimes x_2 \right)
\otimes y_1 \wedge y_2 \quad\text{ and } 
\\
\omega_2&=\left(
-\overline{\tfrac{\partial Q}{\partial x_2}}(\partial x_1,\partial x_2)(Q)  \otimes x_1
+\overline{\tfrac{\partial Q}{\partial x_1}}(\partial x_1,\partial x_2)(Q)  \otimes x_2
\right)
\otimes y_1 \wedge y_2.
\end{aligned}
\end{equation}
Note that $\omega_2$ has $S$-degree $1$ inside $S \otimes V^* \otimes \wedge^2 V^*$.
Also notice that $\omega_2 \neq 0$, since otherwise
$$
\overline{\tfrac{\partial Q}{\partial x_1}}(\partial x_1,\partial x_2)(Q)
=\overline{\tfrac{\partial Q}{\partial x_2}}(\partial x_1,\partial x_2)(Q)
=0
$$
which would then imply the following contradiction to Proposition~\ref{homogeneous-polynomial-fact}:
$$
\begin{aligned} 
0 
=\tfrac{\partial}{\partial x_1}
\left(
{\overline{\tfrac{\partial Q}{\partial x_1}}}(\partial x_1,\partial x_2)(Q)
\right)
+
\tfrac{\partial}{\partial x_2}
\left(\,
\overline{\tfrac{\partial Q}{\partial x_2}}(\partial x_1,\partial x_2)(Q)
 \right)
&=\left(\, \overline{
x_1 \tfrac{\partial Q}{\partial x_1}}+
\overline{x_2  \tfrac{\partial Q}{\partial x_2}}
\, \right)
(\partial x_1,\partial x_2)(Q) \\
&=\deg(Q) \cdot \bar{Q}(\partial x_1, \partial x_2)(Q)\, .
\end{aligned}
$$
 Thus $\omega_2$ is a nonzero element of 
$\left( V^* \ot V^* \ot \wedge^2 V\right)^W$ 
by Lemma~\ref{operatorinvariance} and we may identify $\omega_2$, up to a nonzero scalar,
with $(x_1\ot x_2-x_2\ot x_1)\ot y_1\wedge y_2$ by Lemma~\ref{therecanbeonlyone}.

Finally, we check that $\omega_1$, $\omega_2$ are $K$-linearly independent.
The matrix
expressing $\omega_1$, $\omega_2$  with respect to the $S$-basis 
$x_1 \otimes y_1 \wedge y_2$, $x_2 \otimes y_1 \wedge y_2$
of $S \otimes V^* \otimes \wedge^2 V$ 
has determinant $x_1 \frac{\partial Q}{\partial x_1} + x_2 \frac{\partial Q}{\partial x_2} 
=\deg(Q) \cdot Q \neq 0$.

\end{proof}

\section{Conversion between $(q,t)$-analogues of $h$-vectors and $f$-vectors}
\label{h-to-f-section}
We now highlight some applications of Theorem~\ref{maintheorem}
to $f$-vectors and $h$-vectors in algebraic combinatorics.
We will see in Remark~\ref{answerquestion} later how these applications answer a question
on $q$-Kirkman and $q$-Narayana numbers raised in~\cite[Problem~11.3]{ArmstrongReinerRhoades}.

Let $W$ be a coincidental reflection group with smallest exponent $e_1$ and exponent gap $a$.
In the Introduction, we defined these $(q,t)$-analogues
of the $f$-vector and $h$-vector:
\begin{equation}
\label{analogues}
\begin{aligned}
f_r(W;q,t) 
&=q^{r+a\binom{r}{2}} 
   \qbin{n}{r}{q^a}  
    \frac{(-tq^{e_1};q^a)_{n-r}  (-tq^{-1};q^{-a})_{r}}
               {(q^{e_1+1};q^a)_n}\, 
               &&\ \big(\text{the $(q,t)$-analogue of the $f$-vector}\big)
               ,\\
h_r(W;q,t) &= (-tq^{-ar-1})^{n-r} \qbin{n}{r}{q^a}
                    \frac{ (-tq^{-1};q^{-a})_r } { (q^{e_1+1};q^a)_r }\, 
                    &&\  \big(\text{the $(q,t)$-analogue of the $h$-vector}\big)\, .
 \end{aligned}
\end{equation}
We later explain in Section~\ref{parking-section}
why it is  appropriate to call these
analogues of $f$-vectors and $h$-vectors
(see~\eqref{Cambrian-f-vector} and~\eqref{Cambrian-h-vector}).
In this section, we prove Theorem~\ref{h-to-f-theorem} of the Introduction that
converts between these $(q,t)$-analogues.  As preparation,
we rephrase $f_r(W;q,t)$ using the notation of {\it basic hypergeometric functions}~\cite[Chap.~1]{GasperRahman}:
\begin{equation}
\label{basic-hypergeometric-definition}
{}_{k+1}\phi_k\left[
\left.
\begin{matrix}
a_1 & a_2 & \dots &a_{k+1} \\
      & b_1 & \dots &b_{k} 
\end{matrix}\
 \right|\,  q; z\right]
:= \sum_{r \geq 0} \ z^r\ \frac{ (a_1;q)_r (a_2;q)_r \cdots (a_{k+1};q)_r}{(q;q)_r (b_1;q)_r \cdots (b_k;q)_r}\, .
\end{equation}

\begin{proposition}
\label{2-phi-1-prop}
The following two statements are equivalent reformulations of Theorem~\ref{maintheorem}:
\begin{itemize}[leftmargin=6mm]
\item[]
  $$
\begin{aligned}
  (a)\quad
  f_r(W;q,t)
\ =\
\tfrac{1}{|W|} 
\sum_{w \in W} \frac{\det(1+tw)}{\det(1-qw)}
\ =\ t^r \qbin{n}{r}{q^a} 
               \frac{ (-tq^{e_1};q^a )_{n-r}  (-qt^{-1};q^a)_r }
               { (q^{e_1+1};q^a)_n}\, .
               \hspace{20ex}
\end{aligned}
  $$
\item[]                     
  $$
  \begin{aligned}
  \ (b)\quad
  \sum_{r=0}^n s^r 
f_r(W;q,t)
&\ =\ 
\tfrac{1}{|W|} 
\sum_{w \in W} \frac{\det(1+sw^{-1})\det(1+tw)}{\det(1-qw)}\\
&\ =\ \frac{(-tq^{e_1};q^a)_n} {(q^{e_1+1};q^a)_n} \quad
{}_2 \phi_1 \left[ \left. 
\begin{matrix} 
q^{-an} &  -qt^{-1} \\                              
  & -q^{a(1-n)-e_1}t^{-1} 
\end{matrix}\ \right|\, q^a; -sq^{a-e_1} \right] \, .
\hspace{20ex}
  \end{aligned}
  $$  
\end{itemize}

\end{proposition}

\begin{proof}
For Equation~(a), we simply rewrite Theorem~\ref{maintheorem}${}^\prime$
with 
$$
f_r(W;q,t)
=q^{r+a\binom{r}{2}} 
   \qbin{n}{r}{q^a}  
    \frac{(-tq^{e_1};q^a)_{n-r}  (-tq^{-1};q^{-a})_{r}}
               {(q^{e_1+1};q^a)_n}
=t^r \qbin{n}{r}{q^a} 
               \frac{ (-tq^{e_1};q^a )_{n-r}  (-qt^{-1};q^a)_r }
                      { (q^{e_1+1};q^a)_n}
$$
by applying this easy fact (see~\cite[eqn.~(I.3)]{GasperRahman}) to
the numerator factor $(-tq^{-1};q^{-a})_{r}$:
\begin{equation}
\label{GR-eqnI.3}
(z;q^{-1})_{r}=  z^r \ q^{-\binom{r}{2}}\ (-z^{-1};q)_r\, .
\end{equation} 
To see (b), we reexpress (a) as
$$
f_r(W;q,t)=t^r\ \frac{(q^a;q^a)_n }{(q^{e_1+1};q^a)_n}  \
  \frac{ (-tq^{e_1};q^a)_{n-r} } { (q^a;q^a)_{n-r} } \   \frac{ (-qt^{-1};q^a)_{r} }{(q^a;q^a)_{r}}\, .
$$
  We rewrite the middle quotient in the product by applying the identity
  (see~\cite[(I.11)]{GasperRahman})
$$
\frac{(w;q)_{n-r}}{(z;q)_{n-r}}
= \frac{(w;q)_n}{(z;q)_n} \
   \frac{(q^{1-n}/z;q)_r}{(q^{1-n}/w;q)_r} \
   \left(\frac{z}{w}\right)^r
$$
with $w=-tq^{e_1}$, $z=q^a$ and obtain
$$
f_r(W;q,t)=
\frac{(-tq^{e_1};q^a)_n} {(q^{e_1+1};q^a)_n} \
 (-q^{a-e_1})^r \
\frac{ ( q^{-an} ;q^a)_r }{ (q^a ;q^a)_r}\
 \frac{( -qt^{-1};q^a)_r }{ (-q^{a(1-n)-e_1} t^{-1};q^a)_r }
\, ,
$$
which then immediately gives (b) from definition~\eqref{basic-hypergeometric-definition} of ${}_{2}\phi_1$.
\end{proof}

We can now extend the relation between $f$- and $h$-vectors for $n$-dimensional simple polytopes,
namely,
\begin{equation}
\sum_{r=0}^n s^r f_r  \
=\ \sum_{r=0}^n (1+s)^r \ h_r\, ,
\end{equation}
to a relation between the above $(q,t)$-analogues; we prove the theorem from the Introduction:

\vskip.1in
\noindent
{\bf Theorem~\ref{h-to-f-theorem}.}
{\it
For any coincidental reflection group $W$ with exponent gap $a$,
$$
\sum_{r=0}^n s^r f_r(W;q,t)  
\ =\ \sum_{r=0}^n\, (-sq;q^a)_r \cdot h_r(W;q,t)\, .
$$
}
\begin{proof}
We apply a terminating form of {\it Jackson's ${}_2 \phi_1$-transformation}\footnote{The authors thank Dennis Stanton for pointing them to this identity.}~\cite[III.7]{GasperRahman},
\begin{equation}                                                                
\label{Jackson-transformation}                                                  
{}_2 \phi_1 \left[ \left. 
\begin{matrix} q^{-n} &  b \\
                    & c \end{matrix}\ \right|\, q, z \right]  
\ =\   \frac{(c/b;q)_n}{(c;q)_n}
\ \ {}_3 \phi_2 \left[ \left. 
\begin{matrix} q^{-n} &  b & bzq^{-n}/c\\
                    & bq^{1-n}/c& 0 \end{matrix}\ \right|\, q, q\, \right]\, ,
\end{equation}
to the ${}_2\phi_1$-expression in Proposition~\ref{2-phi-1-prop}(b)
after first replacing $q$ by $q^a$ and then 
selecting
parameter choices
$$
b=-qt^{-1}, \quad c=-q^{a(1-n)-e_1}t^{-1}, \quad z=-sq^{a-e_1}\, .
$$
Since
$$
c/b = q^{a(1-n)-e_1-1},
\quad
bq^{1-n}/c=q^{e_1+1},
\quad
bzq^{-1}/c=-sq,
$$
and
$(z;q)_n=(q^{1-n}/z;q)_n (-z)^n q^{\binom{n}{2}}$
(see~\cite[eqn.~(I.7)]{GasperRahman}), 
one can rewrite the quotient
in~(\ref{Jackson-transformation}) as                                                  
$$
\frac{(c/b;q)_n}{(c;q)_n}
\ =\ \frac{(q^{a(1-n)-e_1-1};a^a)_n}{(-t^{-1}q^{a(1-n)-e_1};q^a)_n} 
\ = \ \frac{(q^{e_1+1};q^a)_n}{(-tq^{e_1};q^a)_n} (-tq^{-1})^n\, .
$$
Thus Proposition~\ref{2-phi-1-prop}(b) and \eqref{Jackson-transformation} imply
that
$$
\begin{aligned}
\sum_{r=0}^n s^r 
f_r(W;q,t)
&\ =\ (-tq^{-1})^n \cdot
{}_3 \phi_2 \left[ \left. 
\begin{matrix} q^{-an} &  -qt^{-1} & -sq\\
                    & q^{e_1+1} & 0 \end{matrix}\ \right|\, q^a, q^a \right]\\
&\ =\ (-tq^{-1})^n\
\sum_{r=0}^n (-sq;q^a)_r \cdot q^{ar} \
    \frac{(q^{-an};q^a)_r\ (-qt^{-1};q^a)_r}
         {(q^a;q^a)_r\ (q^{e_1+1};q^a)_r}\, .  
\end{aligned}
$$
It remains to check that
\begin{equation}
\label{h-re-expression}
(-tq^{-ar-1})^{n-r}\  \qbin{n}{r}{q^a}\
                    \frac{ (-tq^{-1};q^{-a})_r } { (q^{e_1+1};q^a)_r }
\ =\ (-tq^{-1})^n \ q^{ar} \ 
    \frac{(q^{-an};q^a)_r \ (-qt^{-1};q^a)_r}
         {(q^a;q^a)_r \ (q^{e_1+1};q^a)_r}\, .
\end{equation}
We substitute definition~\eqref{q-binomial-definition} 
of the $q$-binomial
and
cancel common factors to see that~\eqref{h-re-expression}
is equivalent to
$$
(-tq^{-ar-1})^{n-r} \ \frac{ (q^a;q^a)_n }{ (q^a;q^a)_{n-r}} \ (-tq^{-1};q^{-a})_r
\ =\
(-tq^{-1})^n \ q^{ar} \ (q^{-an};q^a)_r \ (-qt^{-1};q^a)_r \, .
$$
We verify this last equality by applying~\eqref{GR-eqnI.3} to the factor $(-tq^{-1};q^{-a})_r$ on the 
left and rewriting the factor $(q^{-an};q^a)_r$ on the right using this
fact (from~\cite[I.12]{GasperRahman}) with $q$ replaced by $q^a$:
\begin{equation}
\label{GR-eqnI.12}
(q^{-n};q)_r\ =\ \frac{(q;q)_n}{(q;q)_{n-r}} \ (-1)^r\ q^{\binom{r}{2}-nr}. \qedhere
\end{equation}
\end{proof}

\section{The connection with Catalan, Kirkman, Narayana, Cambrian
and clusters}
\label{parking-section}

We explain here how specializations of 
our product formulas for $f_r(W;q,t)$ and $h_r(W;q,t)$ give 
known product formulas for $q$-Catalan, 
$q$-Kirkman, and $q$-Narayana numbers. We also see how
their $q=1$ specializations give the $f$-vectors
and $h$-vectors for Cambrian and cluster fans.
This starts with certain graded representations of a reflection group
$W$ called {\it graded parking spaces}.

\subsection*{The graded parking spaces}
For a positive integer $p$, define a 
class function $\chi^{(p)}: W \longrightarrow \QQ(q)$ by
$$
\chi^{(p)}(w):=\frac{\det(1-q^pw)}{\det(1-qw)}\, .
$$
For special values of $p$, the function $\chi^{(p)}$ 
is actually a $\ZZ[q]$-valued class function and even turns out to
be the graded character of a genuine $W$-representation.
Ito and Okada~\cite{ItoOkada} tabulated
the values of $p$ for which this holds for each 
irreducible complex reflection group.
For duality groups $W$, these special values of $p$
include all the {\it Fuss-Catalan} cases, that is, 
cases where $p \equiv 1 \bmod{h}$ for $h=d_n=\max\{d_i\}$:
\begin{itemize}
\item This fact is related to work of Haiman~\cite[\S 7]{Haiman} for Weyl groups, 
where  $\chi^{(p)}$ gives a graded version
of the {\it $p$-parking space} $W$-permutation representation on $R/pR$,
in which $R$ is the root lattice for $W$.  
\item 
It holds more generally for real reflection groups $W$ 
via results from representation theory of 
{\it rational Cherednik algebras}, e.g., see~\cite[Remark~4.4]{BessisReiner};
Gordon and Griffeth~\cite[\S 1.6]{GordonGriffeth}
generalize these ideas to all complex reflection groups.
\item  
One may verify this fact from Ito and Okada's 
tabulation~\cite[Table~1]{ItoOkada}, where it only fails for the four 
{\it non-duality} groups $G_{12}$, $G_{13}$, $G_{22}$, and $G_{31}$.
\end{itemize}

\subsection*{The $q$-Kirkman numbers}
In the cases where $\chi^{(p)}$ is a genuine graded character, 
setting $t=-q^p$ in~Proposition~\ref{2-phi-1-prop}(a) gives
an expression for the (graded) multiplicity of the $W$-irreducibles 
$\wedge^r V$ in  $\chi^{(p)}$:
\begin{equation}
\label{coincidental-q-Kirkman-formula}
f_r(W;q,-q^p) 
=\left\langle \chi^{(p)}, \wedge^r(V) \right\rangle_W  \\
=(-1)^r \ q^{pr}\ \qbin{n}{r}{q^a} 
               \frac{ (q^{p+e_1};q^a )_{n-r}  (q^{p-1};q^{-a})_r }
                      { (q^{e_1+1};q^a)_n}\, .
\end{equation}
For real reflection groups $W$, these graded multiplicities are
called {\it $q$-Kirkman numbers}; see~\cite[\S9, \S11]{ArmstrongReinerRhoades}.
Specializing \eqref{coincidental-q-Kirkman-formula} 
to the case of types $A$ (where $e_1=1=a$) and $B/C$ (where $e_1=1, a=2$)
gives 
$$
\begin{array}{rclll}
f_r(A_{n-1};q,-q^p)&=&q^{\binom{r+1}{2}}\  \displaystyle\frac{1}{[p]_q} \
\qbin{n-1}{r}{q} \ \qbin{p+n-r-1}{n}{q} &\text{ for }\gcd(n,p)=1
\text{ and}\\
& & & & \\
f_r(B_n/C_n;q,-q^p)&=&q^{r^2}\ \qbin{\frac{p-1}{2}}{r}{q^2} \ \qbin{\frac{p-1}{2}+n-r}{n-r}{q} &\text{ for }p\text{ odd}.
\end{array}
$$
The special case $p=h+1$ was listed in~\cite{ArmstrongReinerRhoades}.
In verifying that \eqref{coincidental-q-Kirkman-formula}
specializes to these two formulas,
it is helpful to note that $p \equiv 1 \bmod{a}$ in these cases, and
so one can again use \eqref{GR-eqnI.12} to 
rewrite \eqref{coincidental-q-Kirkman-formula}  as 
$$
f_r(W;q,-q^p) 
=q^{r+a\binom{r}{2}}\  \qbin{n}{r}{q^a} \
   \frac{ (q^{p+e_1};q^a )_{n-r} \ (q^a;q^a)_{\frac{p-1}{a}} }
                      { (q^{e_1+1};q^a)_n \ (q^a;q^a)_{\frac{p-1}{a}-r}}\, .
$$

\subsection*{The $q$-Catalans and $q$-Narayanas}
\label{Narayana-subsection}

Setting $s=0$ in Theorem~\ref{h-to-f-theorem} 
gives the identity
\begin{equation}
\label{q-Cat-is-sum-of-q-Narayanas}
\frac{(-tq^{e_1};q^a)}{(q^{e_1+1};q^a)_n}
\ =\ f_0(W;q,t)
\ =\ \sum_{r=0}^n h_r(W;q,t)\, .
\end{equation}
In the cases where $\chi^{(p)}$ is a genuine graded character, 
setting $t=-q^p$ in the left side of \eqref{q-Cat-is-sum-of-q-Narayanas}
gives what one might call the 
{\it $p$-rational version of the $q$-Catalan number} for $W$:
$$
\frac{(q^{p+e_1};q^a)}{(q^{1+e_1};q^a)_n}
\ =\ \prod_{i=1}^n \frac{ 1-q^{p+e_i} }{ 1-q^{1+e_i}}\
=:\ \Cat^{(p)}(W,q).
$$
Here, the first equality assumes $W$ is coincidental 
with smallest exponent $e_1$ and exponent gap $q$.  When
$p =h+1$, this product is called the {\it $q$-Catalan number for $W$};
see, e.g. Armstrong~\cite{Armstrong}.
Thus setting $t=-q^p$ in \eqref{q-Cat-is-sum-of-q-Narayanas} gives
a summation formula for coincidental reflection groups $W$:
$$
\Cat^{(p)}(W,q)
\ =\
\frac{(q^{p+e_1};q^a)}{(q^{e_1+1};q^a)_n}
\ =\ \sum_{r=0}^n h_r(W;q,-q^p),
$$
where each summand has the explicit product formula 
\begin{equation}
\label{coincidental-q-Narayana-formula}
h_r(W;q,-q^p)\ =\ (q^{p-ar-1})^{n-r} \qbin{n}{r}{q^a}
                    \frac{ (q^{p-1};q^{-a})_r } { (q^{e_1+1};q^a)_r }\, .
\end{equation}
Happily, \eqref{coincidental-q-Narayana-formula}
agrees with the type $A$ and $B/C$ formulas for
the {\it $q$-Narayana numbers} computed in~\cite{ReinerSommers}:
\begin{equation}
\label{typeABC-q-Narayana-formulas}
\begin{array}{rclll}
h_r(A_{n-1};q,-q^p)&=&q^{(n-1-r)(p-1-r)}\ \displaystyle\frac{1}{[r+1]_q}\ 
                       \qbin{n-1}{r}{q}\ \qbin{p-1}{r}{q} &\text{ for }\gcd(n,p)=1,\\
h_r(B_n/C_n;q,-q^p)&=&(q^2)^{(n-r)(\frac{p-1}{2}-r)}\ 
\qbin{n}{r}{q^2}\ \qbin{\frac{p-1}{2}}{r}{q^2} &\text{ for }p\text{ odd}.
\rule{0ex}{5ex}
\end{array}
\end{equation}
In verifying this, it is helpful to again use \eqref{GR-eqnI.3}
to replace $(q^{p-1};q^{-a})_r$ in \eqref{coincidental-q-Narayana-formula}
with a multiple of $(q^{1-p};q^a)_r$ and then to further
use that $p \equiv 1 \bmod{a}$ in these cases and
employ \eqref{GR-eqnI.12} again, rewriting 
\eqref{coincidental-q-Narayana-formula} as 
$$
h_r(W;q,-q^p) 
\ =\ q^{(n-r)(p-ar-1)}  \qbin{n}{r}{q^a} \
   \frac{ (q^a;q^a)_{\frac{p-1}{a}} }
                      { (q^{e_1+1};q^a)_n \ (q^a;q^a)_{\frac{p-1}{a}-r}}\ .
$$

\begin{remark}
We say ``happily'' above because the 
formulas in \eqref{typeABC-q-Narayana-formulas} came from a subtle and general 
{\it Weyl group} construction that arose from work of the third author~\cite{Sommers}, 
described in~\cite{ReinerSommers}, that always produces $q$-Narayana numbers
summing to $q$-Catalan numbers.  
But there was nothing, {\it a priori}, indicating that 
they must {\it coincide} with the
values $h_r(W;q,-q^p)$ arising naturally here for each coincidental
reflection group $W$.
\end{remark}
\begin{remark}
\label{answerquestion}
  For the coincidental types, setting $t=-q^p$ in Theorem~\ref{h-to-f-theorem}
  relates the $q$-Kirkman numbers $f_r(W;q,-q^p)$ to the
  $q$-Narayana numbers $h_r(W;q,-q^p)$,
  as asked for in~\cite[Problem~11.3]{ArmstrongReinerRhoades}.
\end{remark}

\subsection*{The $f$-vector and $h$-vector}
\label{cluster-complex-section}
Further specializing to $t=-q^{h+1}$ (so $p=h+1$)
and taking $q \rightarrow 1$
in \eqref{analogues}
produces integers 
for a {\em real} reflection group $W$
\begin{align}
\label{Cambrian-f-vector}
  f_r&:=\left[ f_r(W;q,-q^{h+1}) \right]_{q=1} \text{ and }\\
  \label{Cambrian-h-vector}
h_r&:=\left[ h_r(W;q,-q^{h+1}) \right]_{q=1}
\end{align}
that were observed in~\cite[\S 3.3]{ArmstrongReinerRhoades}  
to be  the {\it $f$-vector} and {\it $h$-vector}, respectively,
for the (finite type) {\it cluster complexes} of Fomin and Zelevinsky (in the Weyl group case) and {\it Cambrian fans} of 
Reading (for arbitrary real reflection groups).   
Specifically, $f_r$ counts the number of clusters of cardinality $n-r$,
or the number of cones in the fan having dimension $n-r$.
Thus Theorem~\ref{h-to-f-theorem} specializes in this instance
to the usual $h$-vector-to-$f$-vector relationship for the
simplicial spheres associated to these fans, or to the simple polytopes
for which they are the normal fans, constructed by
Hohlwedg, Lange, and Thomas \cite{HohlwegLangeThomas}.
This again answers 
the second part of the
question raised in~\cite[Problem~11.3]{ArmstrongReinerRhoades}
for (real) coincidental groups $W$.

\begin{remark}
  Theorem~\ref{maintheorem} explains
  a mysterious product formula observed by Fomin and Reading
  ~\cite[Thm.~8.5 at $m=1$]{FominReading} 
  for the number of $r$-dimensional cones in the cluster/Cambrian fan
  for real coincidental reflection groups $W$:
  $$
   f_r=\binom{n}{r}\prod_{i=1}^{n-r} \frac{h+d_i}{d_i}\, .
  $$
The formula follows from computing~\eqref{Cambrian-f-vector}
by setting $t=-q^{h+1}$ and then $q=1$ in Theorem~\ref{maintheorem}:
    $$
    f_r=\left[ f_r(W;q,-q^{h+1}) \right]_{q=1} =
      \binom{n}{r} \ \frac{\prod_{i=1}^r (h+1-e^*_i) \cdot \prod_{i=1}^{n-r} (h+1+e_i)}
      {\prod_{i=1}^{n} d_i}\
      =\binom{n}{r}\prod_{i=1}^{n-r} \frac{h+d_i}{d_i}
      $$
      where the last equality uses the fact that $d_i=e_i+1$ and that $e_i^*=h-e_{n+1-i}$ for real reflection groups $W$. 
\end{remark}


\section{Data on the non-coincidental exceptional groups}
\label{data-section}
For the {\it non-coincidental} exceptional irreducible reflection groups $W$,
we tabulate here the polynomials 
$$
\nu_r(W,q,t):=
\frac{\Hilb\left((S(V^*)\otimes \wedge V^* \otimes \wedge^r V)^W,q,t \right)}{\Hilb(S(V^*)^W,q)}
$$
for $r=0,1,2,\ldots,n$.  The results of
Section~\ref{G(d,1,n)-section},
including Corollary~\ref{type-A-hook-content-formula},
\ref{mainthmforG(d,1,n)}, \ref{Koike-theorem-wedge-corollary},
give the same data for the Weyl groups of type $A$ and
the infinite family $G(de,e,n)$ of complex reflection groups.
Thus, together with Theorem~\ref{maintheorem}, this completes those calculations
for {\it all} irreducible complex reflection groups.  Also, together
with Remark~\ref{G(de,e,n)-answer-comparison-remark}, it allows
one to check that the answers in the non-coincidental cases
always differ from what Theorem~\ref{maintheorem} would have predicted.

In this tabulation, we may assume without loss of generality that
$1 \leq r \leq n-1$, since 
\begin{align}
\label{Solomon-product}
    \nu_0(W,q,t)&=\prod_{i=1}^n (1+tq^{e_i})\quad\text{ and }\\
\label{consequence-of-Shepler-product}
     \nu_n(W,q,t)&=\prod_{i=1}^n (q^{e_i^*}+t)
\end{align}
for all reflection groups by Solomon's Theorem~\cite{Solomon} and
Theorem~\ref{Shepler-product}, respectively. Additionally,
for duality groups $W$, we may assume $2 \leq r \leq n-1$, since
equation~\eqref{Reiner-Shepler}  (see \cite[eqn.~(2.1)]{ReinerShepler})
implies that
\begin{equation}
  \label{Reiner-Shepler-restated}
\nu_1(W,q,t)=\left(\sum_{i=1}^n q^{e_i} \right) (1+tq^{-1}) \prod_{i=1}^{n-1} (1+tq^{e_i})
\, .
\end{equation}
In general, we will use the notation $[m]_q:=1+q+q^2+\cdots+q^{m-1}$.

\subsection*{Rank $2$ groups}

For rank $2$ complex reflection groups,
one has a formula for $\nu_1(W,q,t)$ from \cite[Cor.~10.2]{ReinerShepler}
$$
\nu_1(W,q,t)=(1+tq^{-1})\big( (q^{e^*_1}+q^{e^*_2}) + t(q^{e_1+1}+q^{e_2+1}) \big).
$$
It was noted there that this agrees with
equation \eqref{Reiner-Shepler-restated}, and hence also with Theorem~\ref{maintheorem},
exactly when $W$ is a rank $2$ duality group, or equivalently, a rank $2$ coincidental group.

\subsection*{Real  but non-coincidental reflection groups of rank at least $3$}

For real reflection groups, we may assume that $2 \leq r \leq \lfloor \frac{n}{2} \rfloor$,
 as they are all duality groups and additionally satisfy
 (see~\cite[Prop.~13.1]{ReinerShepler})
 $$\nu_{n-r}(W,q,t)= t^n\ \nu_{r}(W,q,t^{-1})\, .$$
 
\begin{center}
\begin{tabular}{|c|c|} \hline
  $F_4$ \rule[-1ex]{0ex}{3.5ex}
  &  exponents (1,5,7,11)
  \\\hline\hline
  $\nu_2$ &$(q + t)[2]_{q^4} (1 + t q)   \cdot$
    \rule{0ex}{2.5ex}
\\
  &$\left( (q^5+q^7-q^9+q^{11}+q^{13})(1+t^2)
  + ( 1+q^6+q^8+q^{10}+q^{12}+q^{18})t \right)$
  \rule{0ex}{2.5ex}
  \\ \hline
\end{tabular}
\end{center}

\begin{center}
\begin{tabular}{|c|c|} \hline
  $H_4$\rule[-1ex]{0ex}{3.5ex}
  &  exponents (1,11,19,29) \\\hline\hline
  $\nu_2$ & $(q + t)(1 + t q)   \cdot$
  \rule{0ex}{2.5ex}
  \\
  &$\big(
  ( q^{11} + q^{19} + 2q^{29} + q^{39} + q^{47})(1+t^2)$\\
  &$+ ( 1 + q^{10} + q^{18} + q^{20} + q^{22} + q^{28} + q^{30} + q^{36} + q^{38} + q^{40} + q^{48} + q^{58})t
  \big)$
  \\ \hline
\end{tabular}
\end{center}

\begin{center}
\begin{tabular}{|c|c|} \hline
  $E_6$\rule[-1ex]{0ex}{3.5ex}
  &  exponents (1,4,5,7,8,11) \\\hline\hline
  $\nu_2$ \rule[-2ex]{0ex}{5ex}
  & $(q+t) [3]_{q^3}  (1+q+q^4+q^7+q^8) \prod_{i=1}^3 (1+tq^{e_i}) \cdot$
  \rule{0ex}{3ex}
  \\
  &$ \big( q^4+t\frac{[2]_{q^5}[2]_{q^7}}{[2]_{q}}+t^2 q^7 \big)$
  \rule[-2ex]{0ex}{2ex}
  \\ \hline
  $\nu_3$         \rule[-2ex]{0ex}{5ex}
  &  $[2]_{q^4} [5]_q  \prod_{i=1}^2 (1+tq^{e_i})  \prod_{i=1}^2 (q^{e_i}+t)\cdot$ \\
  & $ \big( q^5(1+t^2)(1-q+q^2+q^3-2q^4+q^5+q^6-q^7+q^8)  $\\
  & $ + t [2]_{q^2}(1-q-q^2+2q^3-q^5+q^6+q^{10}-q^{11}+2q^{13}-q^{14}-q^{15}+q^{16}) \big)$
  \rule[-1.5ex]{0ex}{1ex}
  \\\hline 
 \end{tabular}
 \end{center}
  
\begin{center}
  \begin{tabular}{|c|c|} \hline
    $E_7$\rule[-1ex]{0ex}{3.5ex}
    &   exponents (1,5,7,9,11,13,17) \\\hline\hline
    $\nu_2$ & $(q+t) [3]_{q^6} [7]_{q^2}   \prod_{i=1}^4 (1+tq^{e_i}) \cdot$
    \rule{0ex}{2.5ex}
    \\
    &$ \big( q^5+t\frac{[2]_{q^8}[10]_{q^2}}{[2]_{q^2}}+t^2 q^{11} \big)$
        \rule[-2ex]{0ex}{6ex}\\ \hline
        $\nu_3$ &  $[5]_{q^2} [7]_{q^2}   \prod_{i=1}^3 (1+tq^{e_i}) \prod_{i=1}^2 (q^{e_i}+t)\cdot$
        \rule[-1.5ex]{0ex}{4.5ex}
        \\
        &
        $ \big( (q^7+q^9t^2)\frac{[10]_{q^2}}{[2]_{q^2}} 
            + t[2]_{q^4} (1-2q^2+q^4+q^6-q^{10}+q^{14}+q^{16}-2q^{18}+q^{20})\big)$ \\\hline 
  \end{tabular}
 \end{center}
 
\begin{center}
  \begin{tabular}{|c|c|} \hline
    $E_8$  \rule[-1ex]{0ex}{3.5ex}
    &  exponents (1,7,11,13,17,19,23,29) \\\hline\hline
  $\nu_2$ \rule[-2ex]{0ex}{5ex}
    &$(q+t) (1+q^{12}) \big(-q^{14}+\sum_{i=1}^8 q^{e_i-1} \big)  \prod_{i=1}^5 (1+tq^{e_i}) \cdot$ \\
    &$ \big
    ( (1+q^{12}t^2)((q^7+q^{11})+t(1+q^{14})(1+q^{16})\big)$
    \rule[-1ex]{0ex}{1ex}\\ \hline
    $\nu_3$
    \rule[-2ex]{0ex}{5ex}
    &$[7]_{q^2} \big(\sum_{i=1}^8 q^{e_i-1} \big)   \prod_{i=1}^4 (1+tq^{e_i}) \prod_{i=1}^2 (q^{e_i}+t) \cdot$ \\
&   $\left((1+q^6t^2)(q^{11}-q^{15}+q^{17}-q^{19}+q^{23})+t(1-q^2+q^{12}+q^{28}-q^{38}+q^{40})\right)
    $
    \rule[-1ex]{0ex}{2ex}
    \\ \hline
    $\nu_4$
    \rule[0ex]{0ex}{3ex}
    & $[7]_{q^2} \prod_{i=1}^2 (1+tq^{e_i}) \prod_{i=1}^2 (q^{e_i}+t) \cdot$ \\ 
    & $\Big( (1+t^4)\cdot q^{24}(1 - q^2 + q^4 + q^6 - q^8 + q^{10} + q^{12} - 2q^{14} + 4q^{16} - 2q^{20} $
    \hspace{20ex}
    \\
   \rule[0ex]{0ex}{0ex}
   &\hspace{10ex}
   $+\, 4q^{22} - 2q^{24} + 4q^{28} - 2q^{30} + q^{32} + q^{34} - q^{36} + q^{38} + q^{40} - q^{42} + q^{44}) $ \\
      \rule[0ex]{0ex}{1ex}
      & $ + (t+t^3) \cdot q^{11} [2]_{q^6} \left( 1 - q^4 + q^6 + 3q^{12} - 2q^{14} + q^{16} + q^{18} + q^{22} + 3q^{24} - q^{26} + q^{28} + 2q^{30} \right.$
      \hspace{0ex}
      \\
      & $\left.
      \hspace{25ex}
      + 2q^{34} +  q^{36} - q^{38} + 3q^{40} + q^{42} + q^{46} + q^{48} - 2q^{50} + 3q^{52} + q^{58} - q^{60} + q^{64}\right) $ \\   \rule[0ex]{0ex}{0ex}
      &
         \rule[0ex]{0ex}{0ex}
         $\left.+t^2 \cdot (1 - q^2 + q^6 - q^8 + q^{10} + q^{12} - q^{14} + q^{16} + q^{18} - q^{20} + 3q^{22} + 2q^{24}\right.$
         \hspace{16ex}
         \\
         &$ - 2q^{26} + 5q^{28} + q^{30} + 5q^{34} + 2q^{36} - q^{38} + 9q^{40} - q^{44} + 10q^{46} - q^{48} $
                  \hspace{9ex} \\
                  &$+ 9q^{52} - q^{54} + 2q^{56} + 5q^{58} + q^{62} + 5q^{64} - 2q^{66} + 2q^{68} + 3q^{70} - q^{72}$
                           \hspace{8ex}${}_{}$\\
                           &$  + q^{74} + q^{76} - q^{78} + q^{80} + q^{82} - q^{84} + q^{86} - q^{90} + q^{92} ) \Big)$
                           \hspace{19ex}
                           \\\hline
\end{tabular}
\end{center}

\subsection*{Duality, but non-real and non-coincidental reflection groups of rank at least $3$}

\begin{center}
\begin{tabular}{|c|c|} \hline
  $G_{24}$ \rule[-1ex]{0ex}{3.5ex}
  &  exponents (3,5,13), coexponents (1,9,11) \\\hline\hline
  $\nu_2$\rule[-1.5ex]{0ex}{4.5ex}
  & $(q+t)[3]_{q^2}\big((q^6+t^2)(q^3-q^7+q^9) + t(1-q^2+q^6+q^{12})\big)$   \\ \hline
\end{tabular}
\end{center}

\begin{center}
\begin{tabular}{|c|c|} \hline
  $G_{27}$\rule[-1ex]{0ex}{3.5ex}
  &  exponents (5,11,29), coexponents (1,19,25) \\\hline\hline
  \rule[-1ex]{0ex}{3.5ex}
  $\nu_2$ &$(q + t) \big( (q^{19}+q^5 t^2)(1+q^6+q^{24})+t(1 + q^{18} + 2q^{24} + q^{30} + q^{36})\big)$   \\ \hline
\end{tabular}
\end{center}

\begin{center}
\begin{tabular}{|c|c|} \hline
  $G_{29}$ \rule[-1ex]{0ex}{3.5ex}
  &  exponents (3, 7, 11, 19), coexponents (1, 9, 13, 17) \\\hline\hline
  $\nu_2$ \rule[-2ex]{0ex}{5ex}
  &$(q + t) (1+q^3t) [3]_{q^4} \big( q^7(1+q^{12})(q^2+t^2)+t(1 - q^4 + q^8 + q^{12} + q^{16} + q^{24}) \big)$
    \\ \hline
   $\nu_3$ \rule[-2ex]{0ex}{5ex}
  &$\prod_{i=1}^2 (q^{e_i^*} + t) \cdot \big( (q^{10}+t^2)(\sum_{i=1}^4 q^{e_i^*})+t(1 + q^8)(1+q^{12}+2q^{16}) \big)$   \\ \hline
\end{tabular}
\end{center}

\begin{center}
\begin{tabular}{|c|c|} \hline
  $G_{33}$ \rule[-1ex]{0ex}{3.5ex}
  &  exponents (3, 5, 9, 11, 17), coexponents (1, 7, 9, 13, 15) \\\hline\hline
  $\nu_2$ \rule[-1.5ex]{0ex}{4.5ex}
  &$[5]_{q^{2}} (q + t) \prod_{i=1}^2 (1+q^{e_i}t)\cdot $\\
  & $\big( (q^7+q^9t^2)(1 - q^4 + q^6 + q^8 - q^{10} + q^{12}) +t(1 - q^2 + q^6 + q^{12} + q^{16} + q^{24})\big)$   \\ \hline
  $\nu_3$
  \rule[-1.5ex]{0ex}{4.5ex}
&$[5]_{q^{2}} (1+q^3t) \prod_{i=1}^2 (q^{e_i^*}+t) \cdot $\\
  & $\big( (q^9+q^5t^2)(1 - q^2 + q^4 + q^6 - q^8 + q^{12}) +t[2]_{q^2}[2]_{q^8}(1 - 2q^2 + 2q^4 - q^6 + q^{10}) \big)$
  \rule[-1.5ex]{0ex}{1ex}
  \\ \hline
  $\nu_4$
  \rule[-1.5ex]{0ex}{4.5ex}
  &$[5]_{q^{2}}  \prod_{i=1}^3 (q^{e_i^*}+t) \cdot $\\
  & $\big( (q^{13}+q^3t^2)(1 - q^4 + q^6) +t(1 - q^2 + q^6 + q^{16}))\big)$
  \rule[-1ex]{0ex}{1ex}
  \\ \hline
\end{tabular}
\end{center}

\begin{center}
\begin{tabular}{|c|c|} \hline
  $G_{34}$ \rule[-1ex]{0ex}{3.5ex}
  &  exponents (5, 11, 17, 23, 29, 41), coexponents (1, 13, 19, 25, 31, 37)\\\hline\hline
  $\nu_2$ \rule[-1.5ex]{0ex}{4.5ex}
  & $[5]_{q^6} (q + t) \prod_{i=1}^3 (1 + q^{e_i}t) \cdot$ \\
  & $\big( (q^{13} + q^{31} + q^{43})(1+q^{10}t) + t [2]_{q^{12}}( 1 - q^6 + q^{18} + q^{24} + q^{48})\big)$ \\ \hline
  $\nu_3$ \rule[-2.5ex]{0ex}{6.5ex}
  & $\qbin{5}{2}{q^6}\prod_{i=1}^2 (q^{e_i^*} + t)\prod_{i=1}^2 (1 + q^{e_i}t)\cdot$\\
  & $\big(q^{17} + q^{41}(q^2+t^2) + t(1 - q^6 + q^{18} + q^{24} + q^{36} + q^{48})\big)$ \\ \hline
  $\nu_4$ \rule[-1.5ex]{0ex}{4.5ex}
  & $[5]_{q^6}  (1 + q^5t)\ \prod_{i=1}^3 (q^{e_i^*} + t) \cdot$ \\
  & $\left( (q^{11} + q^{23} + q^{41})(q^{14}+t^2)+t(1 - q^6 + q^{12} + q^{24} + q^{30} + q^{36} + 2q^{48}) \right)$  \\ \hline
  $\nu_5$ \rule[-1.5ex]{0ex}{4.5ex}
  & $[2]_{q^{12}} \prod_{i=1}^4 (q^{e_i^*} + t) \cdot$\\
  & $\left( (q^5 + q^{11} + q^{29})(q^{26}+t^2) + t(1 + q^{18} + q^{24} + q^{36} + q^{42} + q^{48}) \right) $
  \rule[-1ex]{0ex}{1ex}
  \\ \hline
   \end{tabular}
\end{center}

\subsection*{The unique non-duality exceptional reflection group of rank at least $3$}

\begin{center}
\begin{tabular}{|c|c|} \hline
  $G_{31}$ \rule[-1ex]{0ex}{3.5ex}
  &  exponents (7, 11, 19, 23), coexponents (1, 13, 17, 29) \\\hline\hline
  $\nu_1$ \rule[-1.5ex]{0ex}{4.5ex}
  &$[2]_{q^{12}} (q + t)  \prod_{i=1}^2 (1 + q^{e_i}t) \big(1+q^{16} +t(q^{19}+q^{23})\big)$   \\ \hline
  $\nu_2$ \rule[-1.5ex]{0ex}{4.5ex}
  &$  (q + t) (1+q^7t) \Big(
   q^{13} + q^{17} + 2q^{29} + q^{41} + q^{45}$\\
   & $+t( 1 + q^{12} + q^{16} + 2q^{24} + 2q^{28} + 
  2q^{32} + q^{36} + 2q^{40})$\\
  & $+t^2(q^{11} + q^{19} + 2q^{23} + q^{27} + q^{35}) \Big)$ \\ \hline
    $\nu_3$ \rule[-1.5ex]{0ex}{4.5ex}&
    $[2]_{q^{12}} \prod_{i=1}^3 (q^{e_i^*} + t) \big(1+q^{16} +t(q^{7}+q^{11})\big)$     \\ \hline
\end{tabular}
\end{center}


\begin{thebibliography}{99}

\bibitem{Armstrong}
D. Armstrong,
Generalized noncrossing partitions and 
combinatorics of Coxeter groups.
{\it Mem.\ Amer.\ Math.\ Soc.} {\bf 202} (2009), no. 949.


\bibitem{ArmstrongReinerRhoades}
D. Armstrong, V. Reiner and B. Rhoades, 
Parking spaces. 
{\it Adv.\ Math.} {\bf 269} (2015), 647--706. 



\bibitem{BessisReiner}
D. Bessis and V. Reiner,
Cyclic sieving of noncrossing partitions for complex reflection groups. 
{\it Ann.\ Comb.} {\bf 15} (2011), 197--222. 


\bibitem{Briggs}
B. Briggs,
Matrix Factorisations Arising From Well-Generated Complex Reflection Groups,
{\tt arXiv:1704.05966}.

\bibitem{Broue}
M. Brou\'e,
Introduction to complex reflection groups and their braid groups. 
{\it Lecture Notes in Mathematics} {\bf 1988}. 
Springer-Verlag, Berlin, 2010. 


\bibitem{Chevalley}
C. Chevalley,
Invariants of finite groups generated by reflections. 
{\it Amer.\ J.\ Math.} {\bf 77} (1955), 778--782. 

\bibitem{FominReading}
S. Fomin and N. Reading, 
Generalized cluster complexes and Coxeter combinatorics. 
{\it Int.\ Math.\ Res.\ Not.} (2005), no. 44, 2709--2757. 


\bibitem{FominZelevinsky}
S. Fomin and A. Zelevinsky, 
Y-systems and generalized associahedra. 
{\it Ann.\ of Math.} (2) {\bf 158} (2003), 977--1018. 

\bibitem{GasperRahman}
G. Gasper and M. Rahman, 
Basic hypergeometric series, Second edition. 
{\it Encyclopedia of Mathematics and its Applications} {\bf 96}. 
Cambridge University Press, Cambridge, 2004.

\bibitem{GordonGriffeth}
I.G. Gordon, and S. Griffeth, 
Catalan numbers for complex reflection groups. 
{\it Amer.\ J.\ Math.} {\bf 134} (2012), 1491--1502.
 
\bibitem{GrinbergReiner}
D. Grinberg and V. Reiner,
Hopf algebras in combinatorics, {\tt arXiv:1409.8356}.

\bibitem{Haiman}
M.D. Haiman,
Conjectures on the quotient ring by diagonal invariants. 
{\it J. Algebraic Combin.} {\bf 3} (1994), 17--76. 

\bibitem{Doppelgangers}
Z. Hamaker, R. Patrias, O. Pechenik, N. Williams,
Doppelg\"angers: bijections of plane partitions,
{\tt arxiv:1602.05535}.

\bibitem{GyojaNishiyamaShimura}
A. Gyoja, K. Nishiyama, H. Shimura, 
Invariants for representations of Weyl groups and two-sided cells. 
{\it J. Math. Soc. Japan} {\bf 51} (1999), 1--34. 

\bibitem{HochsterEagon}
M. Hochster and J.A. Eagon, 
Cohen-Macaulay rings, invariant theory, and the generic perfection of determinantal loci. 
{\it Amer.\ J.\ Math.} {\bf 93} (1971), 1020–1058.

\bibitem{HohlwegLangeThomas}
  C. Hohlweg, C. Lange, H. Thomas,
  Permutahedra and generalized associahedra.
  {\it Adv.\ Math.} {\bf 226} (2011), 608--640.

\bibitem{ItoOkada}
Y. Ito and S. Okada,
On the existence of generalized parking spaces for complex
reflection groups,
{\tt arxiv.1508.06846}.

\bibitem{Kane}
R. Kane,
Reflection groups and invariant theory,
{\it CMS Books in Mathematics} {\bf 5}, Springer-Verlag, New York, 2001.

\bibitem{KirillovPak}
A.A. Kirillov and I.M. Pak, 
Covariants of the symmetric group and its analogues in A.\ Weil algebras. 
{\it Funktsional.\ Anal.\ i Prilozhen.} {\bf 24} (1990), 9--13; translation in 
{\it Funct.\ Anal.\ Appl.} {\bf 24} (1990), 172--176 (1991).

\bibitem{Koike}
K. Koike, 
Poincar\'e series on symmetric and alternating tensors
for irreducible representations of imprimitive complex reflection groups. 
{\it J. Algebra}  {\bf 169} (1994), 541--551. 

\bibitem{Macdonald}
I.G. Macdonald,
Symmetric functions and Hall polynomials, 2nd edition.
Oxford Classic Texts in the Physical Sciences. 
The Clarendon Press, Oxford University Press, New York, 2015.

\bibitem{MillerFoulkes}
A.R. Miller,
Foulkes characters for complex reflection groups,
{\it Proc.\ AMS} {\bf 143} (2015) 3281--3293. 

\bibitem{MillerWalls}
\bysame,
Walls in Milnor fiber complexes,
{\tt arXiv:1710.03069}.

\bibitem{Molchanov}
V.F. Molchanov, 
Poincar\'e series of representations of finite groups that are generated by reflections.
{\it Funktsional. Anal. i Prilozhen.} {\bf 26} (1992), no. 2, 82--85; translation in 
{\it Funct. Anal. Appl.} {\bf 26} (1992), no.\ 2, 143--145. 

\bibitem{OrlikSolomon}
P. Orlik and L. Solomon,
Unitary reflection groups and cohomology.
{\it Invent.\ Math.} {\bf 59} (1980), 77--94. 

\bibitem{OrlikTerao}
P. Orlik and H. Terao, 
Arrangements of hyperplanes.
{\it Grundlehren der Mathematischen Wissenschaften} {\bf 300}. Springer-Verlag, Berlin, 1992.

\bibitem{Reading}
N. Reading, 
Cambrian lattices. 
{\it Adv.\ Math.} {\bf 205} (2006), no. 2, 313--353. 

\bibitem{ReinerShepler}
V. Reiner and A.V. Shepler,
Invariant derivations and differential forms for reflection groups,
to appear in {\it Proc.\ Lond.\ Math.\ Soc.}; available at {\tt arxiv:1612.01031}.

\bibitem{ReinerSommers}
V. Reiner and E. Sommers,
Weyl group $q$-Kreweras numbers and cyclic sieving,
{\it Ann.\ Comb.} {\bf 22} (2018), 819--874. 

\bibitem{SaitoYanoSekiguchi}
K. Saito, T. Yano, J. Sekiguchi,
On a certain generator system of the ring of invariants of a finite reflection group. 
{\it Comm.\ Algebra} {\bf 8} (1980), 373--408. 

\bibitem{ShephardTodd}
G.C. Shephard and J.A. Todd,
Finite unitary reflection groups. 
{\it Canadian J.\ Math.} {\bf 6}, (1954), 274--304. 

\bibitem{Shepler99}
A. Shepler,
Semi-invariants of finite reflection groups.  
{\em J.\  Algebra} 220 (1999), no.\ 1, 314--326. 


\bibitem{Shepler05}
A. Shepler,
Generalized exponents and forms.
{\it J.\ Algebraic Combin.}, 22 (2005), no.\ 1, 115--132.

\bibitem{SheplerTerao}
A. Shepler and H. Terao,
Logarithmic forms and anti-invariant forms of reflection groups,
In Arrangements – Tokyo 1998, M. Falk and H. Terao, eds. 
(Tokyo: Mathematical Society of Japan, 2000), Adv.\ Stud.\ Pure Math., 273--278.

\bibitem{Solomon}
L. Solomon, 
Invariants of finite reflection groups. 
{\it Nagoya Math.\ J.} {\bf 22} (1963), 57--64. 

\bibitem{Sommers}
E. Sommers, 
Exterior powers of the reflection representation in Springer theory. 
{\it Transform.\ Groups} {\bf 16} (2011), 889--911. 


\bibitem{Stanley1979}
R.P. Stanley,
Invariants of finite groups and their applications to combinatorics,
{\it Bull.\ AMS} {\bf 1} (1979), No.\ 3, 475--511.

\bibitem{Stanley}
R.P. Stanley, 
Relative invariants of finite groups generated by pseudoreflections. 
{\it J.\ Algebra}  {\bf 49} (1977), 134--148. 

\bibitem{Thibon}
J.-Y. Thibon, 
The inner plethysm of symmetric functions and some of its applications. 
{\it Bayreuth. Math. Schr.} {\bf 40} (1992), 177--201. 

\end{thebibliography}
\end{document}